\documentclass[11pt]{amsart}

\newlength{\myhmargin}  
\usepackage[textheight=574pt, textwidth=445pt, marginparwidth=50pt, centering]{geometry}

\usepackage{amsmath,amssymb,amsthm,amsfonts,amscd,flafter,
graphicx,verbatim,pinlabel,mathrsfs,caption, enumitem, dsfont}
\usepackage[all]{xy}
\usepackage{epstopdf}
\usepackage{verbatim}
\usepackage[colorlinks=false]{hyperref}
\usepackage[all]{hypcap}
\usepackage{xcolor}
\usepackage{url}
\AtBeginDocument{\addtocontents{toc}{\protect\setlength{\parskip}{0pt}}}

\newcommand\HFK{\widehat{HFK}}

\newcommand{\longcomment}[2]{#2}

\DeclareFontFamily{U}{mathx}{\hyphenchar\font45}
\DeclareFontShape{U}{mathx}{m}{n}{
      <5> <6> <7> <8> <9> <10>
      <10.95> <12> <14.4> <17.28> <20.74> <24.88>
      mathx10
      }{}
\DeclareSymbolFont{mathx}{U}{mathx}{m}{n}
\DeclareFontSubstitution{U}{mathx}{m}{n}
\DeclareMathAccent{\widecheck}{0}{mathx}{"71}
\newcommand{\HMto}{\widecheck{\mathit{HM}}}

\longcomment{
    \RequirePackage{rotating}                   
    \def\HMto{%
       \setbox0=\hbox{$\widehat{\mathit{HM}}$}
       \setbox1=\hbox{$\mathit{HM}$}
       \dimen0=1.1\ht0
       \advance\dimen0 by 1.17\ht1
       \smash{\mskip2mu\raise\dimen0\rlap{%
          \begin{turn}{180}
              {$\widehat{\phantom{\mathit{HM}}}$}
           \end{turn}} \mskip-2mu    
                \mathit{HM}
    }{\vphantom{\widehat{\mathit{HM}}}}{}}
}
    
    \newcommand*\oline[1]{%
  \vbox{%
    \hrule height 0.35pt
    \kern0.1ex
    \hbox{%
      \kern-0.0em
      \ifmmode#1\else\ensuremath{#1}\fi
      \kern-0.1em
    }
  }
}

\newtheorem{theorem}{Theorem}[section]
\newtheorem{lemma}[theorem]{Lemma}

\newtheorem{corollary}[theorem]{Corollary}
\newtheorem{proposition}[theorem]{Proposition}

\theoremstyle{definition}
\newtheorem{definition}[theorem]{Definition}

\newtheorem{remark}[theorem]{Remark}

\makeatletter
\newtheorem*{rep@thm}{\rep@title}
\newcommand{\newreptheorem}[2]{%
\newenvironment{rep#1}[1][0,0]{%
\def\rep@title{#2##1}%
\begin{rep@thm}}%
{\end{rep@thm}}}
\makeatother
\newreptheorem{theorem}{}
\usepackage{float,mathtools}

\begin{document}

\title{On the Transverse invariant and braid dynamics}

\author[Lev Tovstopyat-Nelip]{Lev Tovstopyat-Nelip}
\address{Department of Mathematics \\ Michigan State University}
\email{tovstopy@msu.edu}


\begin{abstract} 
Suppose $(B,\pi)$ is an open book supporting a contact 3-manifold $(Y,\xi)$, where the binding $B$ is possibly disconnected, and $K$ is a braid about this open book. Then $B\cup K$ is naturally a transverse link in $(Y,\xi)$. We prove that the transverse link invariant in knot Floer homology, \[\widehat{t}(B\cup K)\in \widehat{HFK}(-Y,B\cup K),\] defined in \cite{equiv} is always nonzero. This generalizes the main results of Etnyre and Vela-Vick in \cite{VVbind,torsionob}.
As an application, we show that if $K$ is braided about an open book with connected binding, and has fractional Dehn twist coefficient greater than one, then $\widehat{t}(K)\ne 0$. This generalizes a result of Plamenevskaya \cite{Pla} for classical braid closures. 


 \end{abstract}

\maketitle


\section{introduction}
\label{sec:intro}

A central problem in 3-dimensional contact geometry is to distinguish and classify transverse links. There is a classical  invariant of transverse links called the self-linking number. Distinguishing transverse links that are smoothly isotopic and  have  the same  self-linking number can be quite difficult. In this paper, we study  a transverse link invariant in knot Floer homology which has proven     useful in this regard. We recall this and related invariants below.


Suppose $(B,\pi)$ is an open book supporting a contact 3-manifold $(Y,\xi)$. Any link $K$ braided about the open book is naturally transverse, and all transverse isotopy classes are realized via such braids \cite{pav}. Baldwin, Vela-Vick, and V\'{e}rtesi used this viewpoint in \cite{equiv} to define an invariant of transverse links, the so-called BRAID invariant, which assigns to such  $K$ a class in knot Floer homology, \[\widehat{t}(K)\in \widehat{HFK}(-Y,K).\] They then  proved that $\widehat{t}$ agrees with the  LOSS invariant $\widehat{\mathcal{T}}$ defined by Lisca, Ozsv{\'a}th, Stipsicz, and Szab{\'o}   \cite{LOSS} for transverse knots, and with the  GRID invariant $\widehat{\theta}$ defined by Ozsv{\'a}th, Szab{\'o}, and Thurston  \cite{grid} for   transverse links  in the tight contact $3$-sphere $(S^3,\xi_{std})$.

Our first goal in this paper is to  prove a nonvanishing result for the transverse invariant of a braid together with the binding.
To set the stage, recall that the binding $B$ of an open book $(B,\pi)$ is naturally a transverse link in the supported contact manifold  $(Y,\xi)$.

\begin{theorem}
\label{thm:binding}
Suppose $(B,\pi)$ is an open book supporting $(Y,\xi)$ and $K$ is a transverse link  braided about this open book. The transverse invariant of braid union binding is non-zero, i.e.

\center{$\widehat{t}(B\cup K)\ne 0$}.
\end{theorem}

There are two notable antecedents to   Theorem \ref{thm:binding}. 
First, Vela-Vick proved in \cite{VVbind} that  the LOSS invariant $\widehat{\mathcal{T}}(B)$ is  nonzero when the binding $B$ is connected. Second, Etnyre and Vela-Vick proved in \cite{torsionob} that $\overline{c}(B)$ is nonzero even when $B$ is disconnected, where \[\overline{c}(B)\in SFH(-Y(B),\Gamma_\mu)\cong \HFK(-Y,B)\] is the transverse invariant defined in \cite{StipV} using a partial open book for the link complement. 

Theorem \ref{thm:binding} can be viewed as a generalization of both of these results. Indeed, restricting Theorem \ref{thm:binding} to the case where $K$ is the  empty link and $B$ is connected recovers the first result, since the BRAID and LOSS invariants agree for transverse knots. Moreover, it is shown in \cite{StipV} that $\overline{c}=\widehat{\mathcal{T}}$  for transverse \emph{knots}, and the arguments of \cite{equiv,StipV} should extend to show that $\overline{c}= \widehat{t}$ for  transverse links as well. Therefore, restricting Theorem \ref{thm:binding} to the case where $K$ is empty should recover the second result above.

The dynamical properties of the monodromy $\phi$ of an open book decomposition having connected binding can encode important geometric information about the supported contact structure $\xi_\phi$.
Roughly, the \emph{fractional Dehn twist coefficient} of the monodromy, denoted $c(\phi)$, quantifies how much positive twisting $\phi$ induces near the boundary of the page. 
Honda, Kazez and Mati\'{c} proved that if $c(\phi)<0$, which implies that $\phi$ is not \emph{right-veering}, then the supported contact structure is overtwisted \cite{RVD}; in particular the contact invariant $c(\xi_\phi) \in \widehat{HF}(-Y)$ vanishes. They also established a partial converse, if $c(\phi)\ge1$, then $\xi_\phi$ is weakly symplectically fillable \cite{RVD2}; it follows that $c(\xi_\phi)\ne 0$.
It is interesting to ask how much of this story translates to the setting of braids about open books and the transverse invariant.

Given a \emph{pointed monodromy} $g\in Mod(S\smallsetminus \{p_1,\dots,p_n\},\partial S)$ we may take the closure and obtain a link $K$ braided about the underlying open book. Here we view the surface $S$ as a fixed page and the marked points $\{p_1,\dots,p_n\}$ as $K\cap S$. The notion of fractional Dehn twist coefficient generalizes readily to pointed monodromies. It is shown in \cite{equiv} that if $c(g)<0$ then the BRAID invariant of $K$ vanishes. This was later exploited by Baldwin and Grigsby in \cite{trivialbraid} to give a new solution to the word problem in the Artin braid group. More recently, Plamenevskaya proved the following partial converse for the GRID invariant:



\begin{theorem} 
\label{thm:pla} \cite{Pla}
If $K$ is the closure of $\beta\in B_n$ having $c(\beta)>1$, then $\widehat{\theta}(K)\ne 0$. 
\end{theorem}

The above is proven by applying comultiplication \cite{comultgrid} of the invariant $\widehat{\theta}$ to reduce to the case of studying a reference braid $T$ having $\widehat{\theta}(T)\ne 0$. Reducing to the reference braid relies on the notion of $\sigma$-positivity, coming from the Dehornoy ordering on the braid group. $\widehat{\theta}(T)\ne 0$ is established by constructing a grid diagram for $T$ and studying the associated combinatorial chain complex. It is useful to observe that the reference braid $T$ is transversely isotopic to $K'\cup U$, where $K'$ is some braid about the trivial open book $(U,\pi)$ for $S^3$.

Studying geometric properties of transverse links arising from dynamical properties of their pointed monodromies is more natural from the perspective of the BRAID invariant. Indeed, via an application of Theorem \ref{thm:binding} we are able to generalize the above theorem:

\begin{theorem}
\label{thm:largetwist}
Suppose $K$ is the closure of a pointed monodromy $g\in Mod(S\smallsetminus P,\partial S)$, with $\partial S$ connected. If $c(g)>1$ then the BRAID invariant of $K$ is non-zero, i.e. 
\[\widehat{t}(K)\ne 0.\]
\end{theorem}

The fractional Dehn twist coefficients of a pointed monodromy and that of the underlying open book can differ by arbitrary amounts, see \cite{IKess}. We prove vanishing of the BRAID invariant for loose links (Theorem \ref{thm:loose}) and observe the following immediate corollary:
\begin{corollary}
Suppose $K$ is the closure of a pointed monodromy $g\in Mod(S\smallsetminus P,\partial S)$, with $\partial S$ connected. If $c(g)>1$ then $K$ is non-loose.
\end{corollary}

\noindent The above should be compared with Theorem 6.1 of \cite{IKquasi}, a similar result for braids about planar open books equipped with multiple boundary components.

\subsection{Acknowledgements} We thank John Baldwin for many helpful discussions. 

\section{Preliminaries}

\subsection{Braids and transverse links}
\label{subsec:contact}
We assume the reader has some familiarity with contact geometry. 
For an introduction to the Giroux correspondence consult the notes of Etnyre \cite{openbooks}. For a reference on transverse links we point the reader to \cite{legtransverse}.

Bennequin \cite{ben} studied transverse links in the tight 3-sphere via braids. His ideas have been generalized to study transverse links in arbitrary contact manifolds 3-manifolds.
Suppose that $(B,\pi)$ is an open book supporting a contact 3-manifold $(Y,\xi)$; the binding $B$ is naturally transverse to the contact structure $\xi$.

A link $K\subset Y\smallsetminus B$ is \emph{braided} about $(B,\pi)$ if it is positively transverse to the pages of the open book. 
We say that two braids about an open book are \emph{braid isotopic} if they are isotopic through braids.
Any link braided about an open book is braid isotopic to a link transverse to $\xi$, as the contact planes are nearly tangent to the pages away from a neighborhood of $B$.


There is a notion of positive Markov stabilization for braids with respect to an arbitrary open book, defined in \cite{pav}. This operation increases the braid index by one, but preserves the transverse isotopy class of the braid. The following is a generalization of the transverse Markov theorem of Wrinkle \cite{wrinkle}, and independently Orekov and Shevchishin \cite{orevkov}:

\begin{theorem} \cite{pav}
\label{thm:markov}
Suppose $(B,\pi)$ is an open book supporting $(Y,\xi)$. Every transverse link in $(Y,\xi)$ is transversely isotopic to a braid with respect to $(B,\pi)$.
Moreover, two links braided about $(B,\pi)$ are transversely isotopic if and only if they become braid isotopic after some number of positive Markov stabilizations. 
\end{theorem}



\subsection{Knot Floer homology}
\label{subsec:hfknot}
We fix some notation for knot Floer chain complexes and homology groups. See Section 2 of \cite{equiv} for a review.

To a \emph{weakly admissible} $2n$-pointed Heegaard diagram $\mathcal{H} = (\Sigma,\boldsymbol{\alpha},\boldsymbol{\beta},\bold{w},\bold{z})$, encoding an $m$-component link $L\subset Y$, we associate a chain complex
$CFK^\infty (\mathcal{H})$ \cite{holknots}. As an $\mathbb{Z}_2 [U_1,\dots,U_n]$-module, this complex is generated by tuples of the form
$[\bold{a},i_1,\dots,i_n]$, where
$\bold{a}\in \mathbb{T}_{\boldsymbol{\alpha}}\cap\mathbb{T}_{\boldsymbol{\beta}}$ and each $i_j \in \mathbb{Z}$. The action of the $U$-variables is specified by \[U_j \cdot [\bold{a},i_1,\dots,i_j,\dots,i_n] = [\bold{a},i_1,\dots,i_j -1,\dots, i_n].\]  

The differential counts pseudo-holomorphic disks 
\[
\partial ^\infty [\bold{a},i_1,\dots,i_n] = 
\sum\limits_{ \bold{b}\in \mathbb{T}_{\boldsymbol{\alpha}}\cap\mathbb{T}_{\boldsymbol{\beta} }} \sum\limits_{\substack{\phi\in\pi_2 (\bold{a},\bold{b})\\ \mu(\phi)=1\\  n_\bold{w}(\phi) = 0}}  (\# \widehat{\mathcal{M}}(\phi))  [\bold{b},i_1-n_{z_1}(\phi),\dots,i_n-n_{z_n}(\phi)].
\]
For an integer $c$, the subcomplex $ CFK^{\le c}(\mathcal{H})$ is generated, over $\mathbb{Z}_2 [U_1,\dots,U_n]$, by elements of the form $[\bold{a},c,\dots,c]$. 
The \emph{minus} complex is defined to be $CFK^-(\mathcal{H}) := CFK^{\le 0}(\mathcal{H})$.
If two basepoints $z_i$ and $z_j$ correspond to the same component of $L$, then the variables $U_i$ and $U_j$ act identically on $HFK^-(Y,L):= H_* (CFK^- (\mathcal{H}))$.
By reindexing the basepoints, we arrange that $z_1,\dots, z_m$ correspond to the $m$ distinct components of the link. $HFK^-(Y,L)$ is then an invariant of the pair $(Y,L)$, well-defined up to $\mathbb{Z}_2 [U_1,\dots, U_m]$-module isomorphism. 


By setting $U_1,\dots, U_m = 0$ in $CFK^-(\mathcal{H})$, we obtain the \emph{hat} complex $\widehat{CFK}(\mathcal{H})$. The homology of this complex, denoted $\widehat{HFK}(Y,L)$, is an invariant of the pair $(Y,L)$ up to $\mathbb{Z}_2$-module isomorphism. 
Let $CFK^{\ge 0}(\mathcal{H})$ denote the quotient complex $CFK^\infty (\mathcal{H})/CFK^{\le -1}(\mathcal{H})$. 
For an $\mathbb{Z}_2 [U_1,\dots U_n]$-module $M$, let $M^\vee$ denote 
\[Hom_{\mathbb{Z}_2 [U_1,\dots,U_n]}(M,\mathbb{Z}_2 [U_1,\dots,U_n,U_1 ^{-1},\dots, U_n ^{-1}]/\mathbb{Z}_2 [U_1,\dots,U_n]).\] Observe that 
we can then identify $CFK^{\ge 0}(\mathcal{H}) ^\vee$ with $CFK^- (\Sigma,\boldsymbol{\beta},\boldsymbol{\alpha},\bold{w},\bold{z})$; where $(\Sigma,\boldsymbol{\beta},\boldsymbol{\alpha},\bold{w},\bold{z})$ is a diagram encoding $(-Y,-L)$.

Setting all $U$-variables equal to zero, we obtain the \emph{totally blocked} complex $\widetilde{CFK}(\mathcal{H})$, whose homology $\widetilde{HFK}(\mathcal{H})$ is an invariant of $(Y,L,n)$, where $n$ is the number of basepoint pairs. As $\mathbb{Z}_2$-modules, $\widetilde{HFK}(\mathcal{H})\simeq \widehat{HFK}(Y,L)\otimes W^{\otimes n-m}$, where $W$ is a 2-dimensional vector space. If $Y$ is a $\mathbb{Q}HS^3$ then the two generators of $W$ have bigradings $(0,0)$ and $(-1,-1)$. 
The projection $\widehat{CFK}(\mathcal{H})\to \widetilde{CFK}(\mathcal{H})$ induces an injection $\iota :\widehat{HFK}(Y,L)\xhookrightarrow{} \widetilde{HFK}(\mathcal{H})$.

\subsection{The BRAID invariant}
\label{subsec:HFK}

In this subsection we review the definition of the BRAID invariant. The construction is similar to that of the contact invariant given in \cite{HKM}. 

As before let, $(B,\pi)$ be an open book supporting $(Y,\xi)$. Let $S_\theta:= \pi^{-1}(\theta)$ denote a page, $\phi$ denote the open book monodromy.
Suppose $S$ has genus $g$ and $m$ boundary components. 
Let $K$ be an index $n$ braid about $(B,\pi)$, i.e. suppose $K\cap S_0$ consists of $n$ points, $P=\{p_1,\dots,p_n\}$, and let $g \in Mod(S\smallsetminus P,\partial S)$ denote some lift of $\phi$ having closure $K$.

A \emph{basis of arcs} $\{a_i\}_1 ^{2g+n+m-2}\subset S\smallsetminus P$ is a collection of properly embedded disjoint arcs which cut $S\smallsetminus P$ into $n$ discs, each containing one point of $P$. 
The pointed monodromy $\widehat{\phi}$ together with any basis of arcs specifies a Heegaard diagram $\mathcal{H}_g= (\Sigma,\boldsymbol{\beta},\boldsymbol{\alpha},\bold{w}_K,\bold{z}_K)$ encoding $(-Y,K)$.
Let $\{b_i\}_1^{2g+n+m-2}$ be another basis of arcs obtained by perturbing the endpoints of $a_i$ in the oriented direction of $\partial S$, and isotoping in $S\smallsetminus P$ so that $a_i$ intersects $b_i$ transversely in a single point with positive sign. We construct the diagram $\mathcal{H}$:
\begin{itemize}
\item
$\Sigma=S_{1/2}\cup -S_0$
\item
$\alpha_i = a_i\times \{0,1/2\}$, $\beta_i = b_i\times \{1/2\}\cup g(b_i)\times \{0\}$
\item
$z_i = p_i \times \{0\}$, $w_i = p_i \times \{1/2\}$
\end{itemize}


For each $i$, $\alpha_i$ intersects $\beta_i$ in a single point in the region $S_{1/2}$ denoted $x_i$. Let $\bold{x} \in \mathbb{T}_{\boldsymbol{\beta}} \cap \mathbb{T}_{\boldsymbol{\alpha}}$ denote the generator\footnote{in the $CFK^-$-complex, we mean the generator $[\bold{x},0,\dots,0]$, in the notation of the previous sub-section.} having component $x_i$ on $\alpha _i$. The homology class $[\bold{x}] \in HFK^\circ(-Y,K)$ is an invariant of the transverse isotopy class of $K$, denoted $t^\circ(K)$ (Theorem 3.1 of \cite{equiv}); where $\circ\in\{\wedge,-\}$. The proof of invariance utilizes the Giroux correspondence \cite{giroux} and Theorem \ref{thm:markov}. Any two braidings of a transverse link $K$ about two open books are related by a sequence of braid isotopies, positive Markov stabilizations of the braids, and positive Hopf stabilizations of the open books. Each of these moves, along with arc slides relating bases of arcs, give rise to a quasi-isomorphisms of the underlying $CFK^\circ$ complexes preserving the distinguished generator $\bold{x}$, see \cite{equiv} for details. 

Given the pointed monodromy $g$, we can also consider the homology class of $\bold{x} \in \widetilde{CFK}(\mathcal{H})$, denoted $\widetilde{t}(g)\in \widetilde{HFK}(\mathcal{H})$. The proof of arc-slide invariance for the invariant $t^\circ$ can be used to show that $\widetilde{t}$ is an invariant of the pointed monodromy $g$. 

\begin{lemma}
\label{tildenonvanishing}
Suppose $g \in Mod(S\smallsetminus P,\partial S)$ has transverse closure $K_g\subset (Y_g,\xi_g)$. Then
\[
\widehat{t}(K_g)\ne 0 \iff \widetilde{t}(g)\ne 0.
\]
\end{lemma}
\begin{proof}
The natural injection $\iota: \widehat{HFK}(-Y_g,K_g)\xhookrightarrow{} \widetilde{HFK}(\mathcal{H})$ maps $\widehat{t}(K_g)$ to $\widetilde{t}(g)$.
\end{proof}



\subsection{Right-veering mapping classes and the fractional Dehn twist coefficient}
\label{subsec:FDTC}

The monoid of \emph{right-veering} diffeomorphisms in the mapping class group, along with the notion of \emph{fractional Dehn twist coefficient}, was introduced and studied by Honda, Kazez, and Mati\'{c}. We do not give the definition of the fractional Dehn twist coefficient, see sections $2$ and $3$ of \cite{RVD} for details.

Let $S$ denote an oriented compact surface with boundary.
Let $a,b\subset S\smallsetminus P$ be oriented properly embedded arcs such that $\partial a = \partial b \subset a\cap \partial S = b\cap \partial S$, as oriented 0-manifolds.
We apply an isotopy of the arcs, fixing their endpoints, so that they intersect transversally in the fewest number of possible intersections; in this case we say the arcs intersect \emph{efficiently}.
The arc $b$ is \emph{to the right of} $a$, denoted $a\le b$, if the two arcs can be made to intersect efficiently so that the tangent vectors $(\dot{b},\dot{a})$ form a positive basis for $T_p(S)$ at both points $p\in\partial a$. We say that $\psi\in Mod(S\smallsetminus P,\partial S)$ is \emph{right-veering} if $\psi(a)\ge a$ for every properly embedded arc $a\subset S\smallsetminus P$ meeting $\partial S$.

Suppose now that $\partial S$ is connected. We denote the \emph{fractional Dehn twist coefficient} of $\psi \in Mod(S\smallsetminus P,\partial S)$ by $c(\psi)$. The definition involves the Nielsen-Thurston classification of surface diffeomorphisms. 
Although the fractional Dehn twist coefficient was originally defined in the case that $P=\emptyset$, the definition carries over naturally when marked points are introduced.

Let $\tau_{\partial S}\in Mod(S\smallsetminus P,\partial S)$ denote a boundary parallel right-hand Dehn twist.
The following basic properties of the twist coefficient follow from the definition:

\begin{lemma} (Section 3 of \cite{RVD} and Proposition 4.10 \cite{IKess})
\label{lem:boundarytwisting}

\begin{itemize}
\item $c(\tau_{\partial S}^N \circ \psi) = N+ c(\psi)$
\item If $c(\psi)>0$ then $\psi$ is right-veering.
\end{itemize}
\end{lemma}


\subsection{Loose transverse links}
\label{subsec:loose}
A transverse link $K\subset (Y,\xi)$ is \emph{loose} if its complement $Y\smallsetminus K$ contains an overtwisted disk. We prove that the BRAID invariant of a loose transverse link vanishes by realizing it as the closure of a non right-veering pointed monodromy. This is similar to Theorem 1.4 of \cite{LOSS}.

\begin{theorem}
\label{thm:loose}
If $K\subset (Y,\xi)$ is a loose transverse link then $\widehat{t}(K)=0$.
\end{theorem}
\begin{proof}
By Lemma 5.1 of \cite{LOSS} we may decompose $(Y,\xi)$ as a contact connected sum $(Y_1,\xi_1)\# (S^3,\xi_2)$ such that $K\subset (Y_1,\xi_1)$ and $\xi_2$ is overtwisted.  
Let $g\in Mod(S \smallsetminus P,\partial S)$ denote a pointed monodromy specifying $K\subset (Y_1,\xi_1)$. 
Theorem 1.1 of \cite{RVD} guarantees the existence of an abstract open book $(\Sigma,\phi)$ supporting $(S^3,\xi_2)$ which is not right-veering.

Let $S\sharp \Sigma$ denote the boundary connected sum of surfaces. There is a natural pointed monodromy $h\in Mod\big{(}(S \smallsetminus P) \sharp\Sigma,\partial (S\sharp\Sigma)\big{)}$ which restricts to $g$ and $\phi$ on each of the two respective summands. By construction, the transverse closure of $h$ is $K\subset (Y,\xi)$, and $h$ is not right-veering. Theorem 1.4 of \cite{equiv} implies $\widehat{t}(K)=0$.

\end{proof}

\subsection{Transversely braiding the binding of an open book}
\label{subsec:braidingbinding}

We fix some notation and illustrate how to transversely braid the binding of an open book about itself.

Let $P = \{p_1,\dots,p_n\}$ denote a collection of marked points in the interior of an oriented surface $S$. 
Let $\gamma\subset S\smallsetminus \{P\cup \partial S\}$ denote a properly embedded arc connecting two points $p_i\ne p_j$ of $P$. 
We let $\sigma_\gamma$ denote the \emph{right-hand half twist about} $\gamma$; $\sigma_\gamma$ has support in a small neighborhood of $\gamma$, see Figure \ref{fig:halftwist}. 
If $\delta$ is a simple closed curve in $S\smallsetminus \{P\cup \partial S\}$, we let $\tau_\delta$ denote a \emph{right-hand Dehn twist about} $\delta$.

We will also make use of push-maps. Let $r$ denote an oriented properly embedded arc $\gamma\subset S\smallsetminus \{P\cup \partial S\}$ having endpoints at a single point $p\in P$. 
Let $\delta_1\cup \delta_2$ denote the oriented boundary of a small annular neighborhood of $r$, where the orientation of $\delta_1$ agrees with that of $r$.
The \emph{push-map} of $p$ along $r$, denoted $\pi_r$, is the composition $\tau_{\delta_1} \circ \tau_{\delta_2}^{-1}$. 
This map has a simple interpretation: the oriented arc $r$ specifies a movie of how the component of the closure of $\pi_r$ corresponding to the marked point $p$ intersects the pages of the open book as one flows through the fibration. Note that if one forgets the basepoint $p$, then the right and left-hand Dehn twists cancel, so $\pi_r$ does not change the underlying open book.

\begin{figure}[h]
\def\svgwidth{140pt}
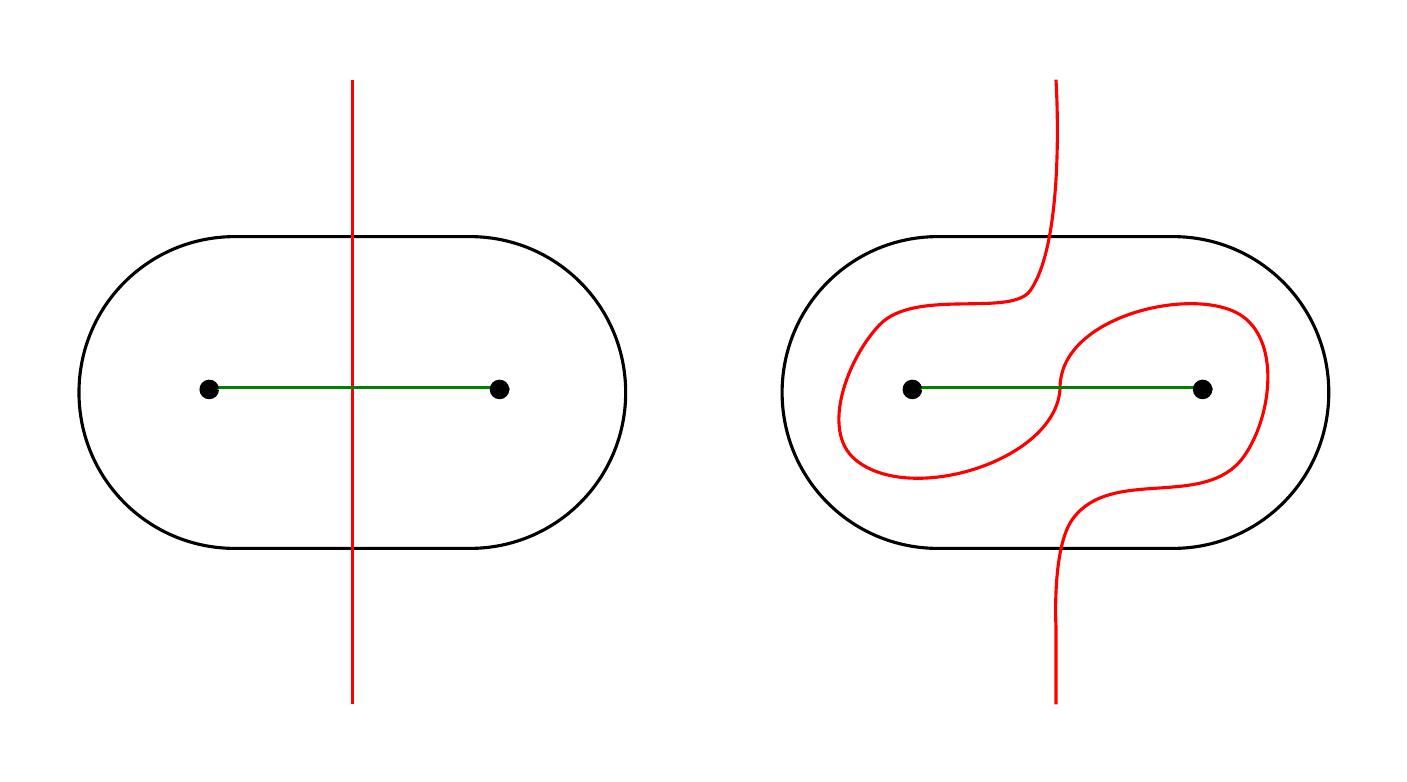
\caption{The right-hand half twist $\sigma_\gamma$ is supported in a neighborhood of the horizontal arc $\gamma$ and is determined by where it sends the vertical arc.}
\label{fig:halftwist}
\end{figure}

The binding $B$ of an open book supporting $(Y,\xi)$ is naturally a transverse link. We illustrate how to braid $B$ about $(B,\pi)$ via a transverse isotopy.

Any transverse knot $K$ admits a neighborhood contactomorphic to 
\[
N_\epsilon = \{(r,\theta,z):r<\epsilon\} \subset \mathbb{R}^2 \times S^1
\]
where $\xi = ker(\alpha) = ker(dz+r^2 d\theta)$, and $K$ is identified with $(0,0)\times S^1$. In these coordinates, $K$ admits a parametrization $\gamma (t) = (0,t,t)$, where $t\in [0,2\pi)$. Consider the transverse isotopy
\[
\Gamma _s (t) = (s,t,t)
\]
from $\gamma _0 (t)$ to $\gamma_{\epsilon /2} (t)$. Applying this isotopy to each component of $B$ realizes the component as an index 1 braid about $(B,\pi)$. 

Suppose $B$ has $n$ components. Let $(S,\phi)$ be an abstract open book corresponding to $(B,\pi)$, where $S$ has genus $g$. 
Abusing notation, we denote the index-$n$ braiding of $B$ obtained by the above procedure by $B$ as well. $B$ is specified by a lift $\widehat{\phi} \in Mod( S\smallsetminus \{p_1,\dots,p_n\},\partial S)$ of $\phi$. Thinking of $\phi$ as fixing a collar neighborhood $\nu(\partial S)$ of the boundary, one obtains $\widehat{\phi}$ by adding $n$ marked points and composing $\phi$ with $n$ push maps supported in $\nu (\partial S)$, each along an arc parallel to and oriented as the nearby component of $\partial S$. See Figure \ref{fig:one} for the push maps and a basis of arcs $\{a_i\}_1^{2g+n-1}\cup \{a_{2,i}\}_{2g+1}^{2g+n-1}$ for $S\smallsetminus \{p_1,\dots,p_n\}$.

\begin{figure}[h]

\def\svgwidth{170pt}
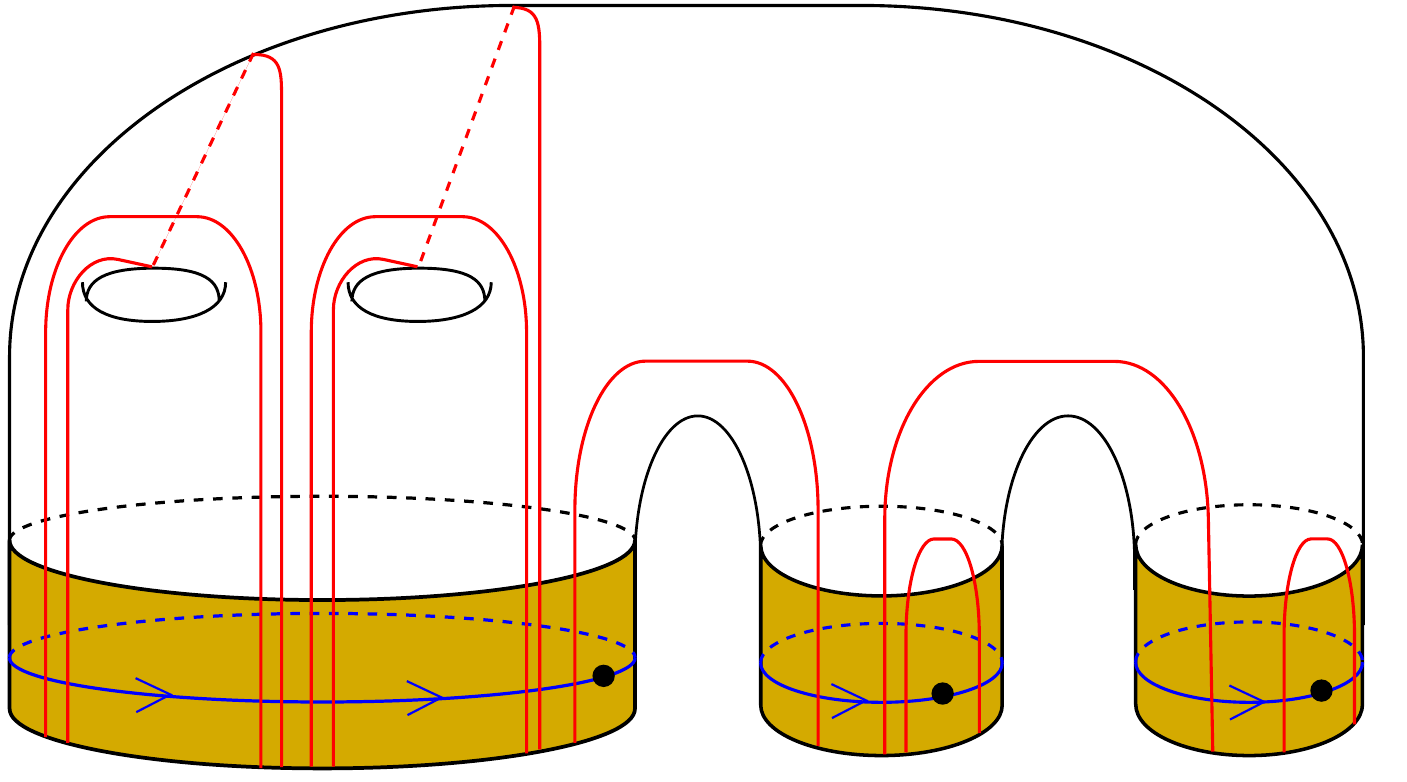
\caption{The basis of arcs $\{a_i\}_1^{2g+n-1}\cup \{a_{2,i}\}_{2g+1}^{2g+n-1}$ for the case $n=3$ and $g=2$ is depicted. The push maps along boundary parallel curves, supported in the shaded neighborhood $\nu (\partial S)$, are depicted as well.}
\label{fig:one}
\end{figure}

\begin{remark}
\label{remark:a}
The construction (of $\widehat{\phi})$ above can be adapted in a straight forward way, with only minor changes in notation, to give the following: Suppose $\partial S$ is connected.
Given a pointed monodromy $g\in Mod(S\smallsetminus \{p_1,\dots,p_n\},\partial S)$ having transverse closure $K$, one can construct a lift $\widehat{g}\in Mod(S\smallsetminus \{p,p_1,\dots,p_n\},\partial S)$ of $g$ having transverse closure $B\cup K$. The new point $p$ is the intersection of $B$, as an index one braid, with a page of the open book.
\end{remark}

The following lemma allows us to recognize certain instances when the closure of a pointed monodromy is transversely isotopic to braid union binding.
\begin{lemma}
\label{lemma:recognition}
Let $S$ be an oriented surface with connected boundary. Suppose that \[h\in Mod(S\smallsetminus \{p,p_1,\dots,p_n\},\partial S)\] fixes an annular neighborhood of $\partial S$ containing only the marked point $p$, then the closure of 
\[
\tau_{\partial S} \circ h
\]
is transversely isotopic to $B\cup K$, where $B$ is the binding of the underlying open book (of $\tau_{\partial S}\circ h$) and $K$ is some braid about this open book.
\end{lemma}
\begin{proof}

Since the pointed monodromy $h$ fixes an annular neighborhood of $\partial S$ containing $p$, it restricts to an element $g\in Mod(S\smallsetminus \{p_1,\dots,p_n\},\partial S)$.
By construction of $\widehat{\tau_{\partial S}\circ g}$ (see Remark \ref{remark:a}), and the definition of a push-map, it follows that $\tau_{\partial S}\circ h = \widehat{\tau_{\partial S}\circ g}$. 
\end{proof}

\subsection{Elementary non-vanishing results}
We establish some non-vanishing results for the BRAID invariant $\widehat{t}$ which play key roles in the proofs of Theorems \ref{thm:binding} and \ref{thm:largetwist}. 
\begin{proposition}
\label{prop:nonzero}
Let $S$ be an oriented, compact surface with boundary and some marked points. Let $K_\gamma$ denote the closure of a right-hand half twist along some embedded arc $\gamma$ connecting two marked points. Let $K_\delta$ denote the closure of a right-hand Dehn twist along any simple closed curve $\delta$. Then $\widehat{t}(K_\gamma),\widehat{t}(K_\delta) \ne 0.$

\end{proposition}
\begin{proof}

We first prove that $\widehat{t}(K_\gamma)$ is non-zero. One can choose a basis of arcs $\{a_i\}$ for $S\smallsetminus P$ such that $a_1$ intersects $\gamma$ in a single point, and all other basis arcs $\{a_i|i>1\}$ are disjoint from $\gamma$. We use this basis of arcs to obtain a Heegaard diagram $\mathcal{H}$ encoding $(-Y,K_\gamma)$ as in Subsection \ref{subsec:HFK}. The resulting complex $\widehat{CFK}(\mathcal{H})$ has vanishing differential, so the result follows.

Next consider we consider a right-hand Dehn twist $\tau_\delta$ along a simple closed curve $\delta$. If $\delta$ is non-separating in $S$, we may construct a basis of arcs $\{a_i\}$ for $S\smallsetminus P$ such that $a_1$ intersects $\delta$ in a single point, and all other basis arcs $\{a_i|i>1\}$ are disjoint from $\delta$. The $\widehat{CFK}$ complex coming from the corresponding diagram will again have vanishing differential.

We are left to consider the case that $\delta$ separates a region $R\subset S$ from $\partial S$. 
Consider the surface $R'$ obtained by removing some small neighborhoods $D_i$ of the points in $p_i\in P\cap R$.
By the chain relation (see Proposition 4.12 of \cite{primer}), $\tau_\delta$ can be expressed as a product of right-hand Dehn twists about non-separating curves in $R'$, along with left-hand Dehn twists about the boundaries of the disks $D_i$. The latter are trivial as elements of the pointed mapping class group. Theorem \ref{thm:comultiplication} and the previous case imply $\widehat{t}(K_\delta)\ne 0$.

\end{proof}


\section{Comultiplicativity of the BRAID invariant}
\label{sec:comultiplicaiton}

We establish some comultiplicativity statements for the BRAID invariant. The proof is an adaptation of the argument for the comultiplicativity of the contact class \cite{comultcontact}, and also generalizes the comultiplicativity of the GRID invariant \cite{comultgrid}.


Let $S$ be a surface with boundary and $P = \{p_1,\dots,p_n\}\subset S$ be a collection of marked points. Let $K_g\subset (Y_g,\xi_g), K_h\subset (Y_h,\xi_h)$, and $K_{hg}\subset (Y_{hg},\xi_{hg})$ be transverse links specified by elements $g,h\in Mod(S\smallsetminus P,\partial S)$ as in Subsection \ref{subsec:contact}.


\begin{theorem}
\label{thm:comult}
There exists a natural comultiplication homomorphism 
\[
\mu^-_*:HFK^-(-Y_{hg},K_{hg})\to HFK^-(-Y_{g},K_{g})\otimes HFK^-(-Y_{h},K_{h}),
\]
of $\mathbb{Z}_2 [U_1,\dots,U_n]$-modules, sending $t^-(K_{hg})$ to $t^-(K_g) \otimes t^-(K_h)$. 
\end{theorem}
\begin{remark}
Recall that the $U$-actions corresponding to the same component of the link act identically on $HFK^-$, however when composing pointed monodromies it is usually not practical to keep track of this. In practice, this theorem is more useful when one quotients $HFK^-$ by the action of $(U_i-U_j)$ for every $i,j$. This has the effect of identifying all the $U$-variable actions and endows the result with the structure of a $\mathbb{Z}_2 [U]$-module. The map $\mu^-_*$ of the Theorem then induces a map of these $\mathbb{Z}_2 [U]$-modules. Alternatively, one can choose to work with the collapsed knot Floer homology, where all of the variables are identified on the chain level; the proof carries through in this setting.

\end{remark}

Let $\{a_i\}$ and $\{b_i\}$ denote the bases of arcs described in Subsection \ref{subsec:HFK}. We construct a third basis of arcs $\{c_i\}$ from by applying a small isotopy moving the endpoints of arcs in $\{b_i\}$ along $\partial S$ in the direction given by the boundary orientation. We require that both $a_i$ and $b_i$ intersect $c_i$ transversely in a single point with positive sign.
We construct three sets of curves and two sets of basepoints on the Heegaard surface $\Sigma = S_{1/2}\cup -S_0$:

\begin{itemize}
\item
$\boldsymbol{\alpha} = \{\alpha_i\} = \{a_i\times \{0,1/2\}\}$ 
\item
$\boldsymbol{\beta} = \{\beta_i\} = \{b_i\times \{1/2\}\cup g(b_i)\times \{0\}\}$
\item
$\boldsymbol{\gamma} = \{\gamma_i\} = \{c_i\times \{1/2\}\cup h(g(c_i)))\times \{0\}\}$
\item
$\bold{w} = \{w_i\} = \{p_i \times \{1/2\}\} \text{ and }\bold{z} = \{z_i\} = \{p_i \times \{0\}\}$.
\end{itemize}

\noindent The Heegaard diagrams $\mathcal{H}_g= (\Sigma, \boldsymbol{\beta},\boldsymbol{\alpha},\bold{w},\bold{z})$, $\mathcal{H}_h= (\Sigma, \boldsymbol{\gamma},\boldsymbol{\beta},\bold{w},\bold{z})$ and $\mathcal{H}_{hg}= (\Sigma, \boldsymbol{\gamma},\boldsymbol{\alpha},\bold{w},\bold{z})$ encode $(-Y_g,K_g)$, $(-Y_h,K_h)$ and $(-Y_{hg},K_{hg})$, respectively. There are unique generators $\bold{x}_g \in \mathbb{T}_{\boldsymbol{\alpha}}\cap \mathbb{T}_{\boldsymbol{\beta}}$, $\bold{x}_h\in \mathbb{T}_{\boldsymbol{\beta}}\cap \mathbb{T}_{\boldsymbol{\gamma}}$ and $\bold{x}_{hg}\in \mathbb{T}_{\boldsymbol{\alpha}}\cap \mathbb{T}_{\boldsymbol{\gamma}}$ having all components in the region $S_{1/2}$. 
By construction, the homology classes of $\bold{x}_g,\bold{x}_h$ and $\bold{x}_{hg}$ 
are the braid invariants $t^\circ(K_g), t^\circ(K_h)$ and $t^\circ(K_{hg})$, respectively. 

There is a version of comultiplication for the totally blocked theory, which is what we use for the proof of Theorem \ref{thm:largetwist}. We prove this result then return to the proof of Theorem \ref{thm:comult}.
\begin{theorem}
\label{thm:comultiplication}
There exists a natural comultiplication homomorphism
\[
\widetilde{\mu}_*: \widetilde{HFK}(\mathcal{H}_{hg})\to \widetilde{HFK}(\mathcal{H}_g)\otimes \widetilde{HFK}(\mathcal{H}_h),
\]
sending $\widetilde{t}(hg)$ to $\widetilde{t}(g)\otimes \widetilde{t}(h)$.
In particular, $\widehat{t}(K_g),\widehat{t} (K_h) \ne 0$ implies that $\widehat{t}(K_{hg})\ne 0$.
\end{theorem}
\begin{proof}
The statement involving the $\widehat{t}$ invariant will follow from the comultiplicativity of the totally blocked invariant, by Lemma \ref{tildenonvanishing}.
For a Heegaard diagram $\mathcal{H}$, let $-\mathcal{H}$ denote the diagram obtained by interchanging the two sets of attaching curves, e.g. $-\mathcal{H}_g = (\Sigma, \boldsymbol{\alpha},\boldsymbol{\beta},\bold{w},\bold{z})$.
If the Heegaard triple diagram $(\Sigma,\boldsymbol{\alpha},\boldsymbol{\beta},\boldsymbol{\gamma},\bold{w},\bold{z})$ is \emph{weakly-admissible} there is a multiplication chain map
\[
\widetilde{m}: \widetilde{CFK}(-\mathcal{H}_g)\otimes \widetilde{CFK}(-\mathcal{H}_h) \to \widetilde{CFK}(-\mathcal{H}_{hg})
\]
defined on the generators by counting pseudo-holomorphic triangles
\[
\widetilde{m}(\bold{a}\otimes \bold{b}) = \sum\limits_{ \bold{x}\in \mathbb{T}_{\boldsymbol{\alpha}}\cap\mathbb{T}_{\boldsymbol{\gamma} }} \sum\limits_{\substack{\phi\in\pi_2 (\bold{a},\bold{b},\bold{x})\\ \mu(\phi)=0\\ n_\bold{z}(\phi) = n_\bold{w}(\phi) = 0}}  (\# \widehat{\mathcal{M}}(\phi)) \bold{x}.
\]
Applying the $\text{Hom}_{\mathbb{Z}_2} (-,\mathbb{Z}_2)$ functor to the above yields a comultiplication chain map
\[
\widetilde{\mu}: \widetilde{CFK}(\mathcal{H}_{hg}) \to \widetilde{CFK}(\mathcal{H}_g)\otimes \widetilde{CFK}(\mathcal{H}_h)
\]
defined on generators by 

\[
\widetilde{\mu}(\bold{x}) = \sum\limits_{\substack{\bold{a}\in\mathbb{T}_{\boldsymbol{\alpha} \cap\mathbb{T}_{\boldsymbol{\beta}}}}\text{ }\\
 {\bold{b}\in\mathbb{T}_{\boldsymbol{\beta} \cap\mathbb{T}_{\boldsymbol{\gamma}}}}}
 \sum\limits_{\substack{\phi\in\pi_2 (\bold{a},\bold{b},\bold{x})\\ \mu(\phi)=0\\ n_\bold{z}(\phi) = n_\bold{w}(\phi) = 0}}  (\# \widehat{\mathcal{M}}(\phi)) \bold{a}\otimes \bold{b}.
\]

To prove Theorem \ref{thm:comultiplication} it suffices to show that the triple diagram above is weakly admissible, and that $\mu (\bold{x}_{hg}) = \bold{x}_g\otimes \bold{x}_h$. For the latter, we argue there is a unique pair of generators $(\bold{a}\in \mathbb{T}_{\boldsymbol{\alpha}}\cap \mathbb{T}_{\boldsymbol{\beta}}, \bold{b} \in \mathbb{T}_{\boldsymbol{\beta}}\cap \mathbb{T}_{\boldsymbol{\gamma}})$ for which there exists a Whitney triangle $\phi \in \pi_2 (\bold{a},\bold{b},\bold{x}_{hg})$ satisfying $n_\bold{w} (\phi) = 0$ admitting a holomorphic representative. We will see that $\bold{a} = \bold{x}_g, \bold{b} = \bold{x}_h$, that the Whitney triangle is unique, and that the triangle has a single holomorphic representative contributing to $\widetilde{m}(\bold{a}\otimes \bold{b})$.

\begin{definition}
Let $D_1,\dots,D_m$ denote the connected regions of $\Sigma \smallsetminus \{\boldsymbol{\alpha}\cup\boldsymbol{\beta}\cup\boldsymbol{\gamma}\}$. A \emph{triply-periodic domain} of a pointed Heegaard triple diagram $(\Sigma,\boldsymbol{\alpha},\boldsymbol{\beta},\boldsymbol{\gamma},\bold{z},\bold{w})$ is a formal linear combination $P = \sum_i p_i D_i$ such that $ n_\bold{w}(P) = 0$ and $\partial P = \sum_i p_i\partial D_i$ is a linear combination of complete $\boldsymbol{\alpha},\boldsymbol{\beta}$ and $\boldsymbol{\gamma}$ curves.
\end{definition}

\begin{definition}
A pointed Heegaard diagram is said to be \emph{weakly-admissible} if every non-trivial triply-periodic domain $P$ has both positive and negative coefficients.
\end{definition}

For each $i$, the triple of curves $\alpha_i$, $\beta_i$ and $\gamma_i$ intersect form the arrangement on $S_{1/2}$ pictured in Figure \ref{fig:admissibility}.
Suppose that $P = \sum_i p_i D_i$ is a triply periodic domain of $(\Sigma,\boldsymbol{\alpha},\boldsymbol{\beta},\boldsymbol{\gamma},\bold{z},\bold{w})$. The regions $D_1$ and $D_6$ will always contain points of $\bold{w}$ (for some values of $i$ the regions $D_1$ and $D_6$ coincide), therefore $p_1 = p_6 = 0$. Since $\partial P$ contains some total number of $\boldsymbol{\alpha}$ curves, we have that $p_7 = p_4-p_5 = -p_2$. If $p_7 \ne 0$ we have that the domain $P$ has both positive and negative coefficients, so assume that $p_7 = p_2 =0$ and $p_4=p_5$. Because $\partial P$ contains some total number of $\boldsymbol{\beta}$ curves it now follows that $p_4 =p_5 = -p_3$. Either $P$ is to have both positive and negative coefficients, or we have that $p_1 = p_2 = \dots = p_7 = 0$. Since all $\boldsymbol{\alpha}$, $\boldsymbol{\beta}$ and $\boldsymbol{\gamma}$ curves meet the region $S_{1/2}$ it is clear that any non-trivial triply periodic domain must have both positive and negative coefficients. We conclude that the diagram is weakly-admissible.

\begin{figure}[h]
\def\svgwidth{180pt}
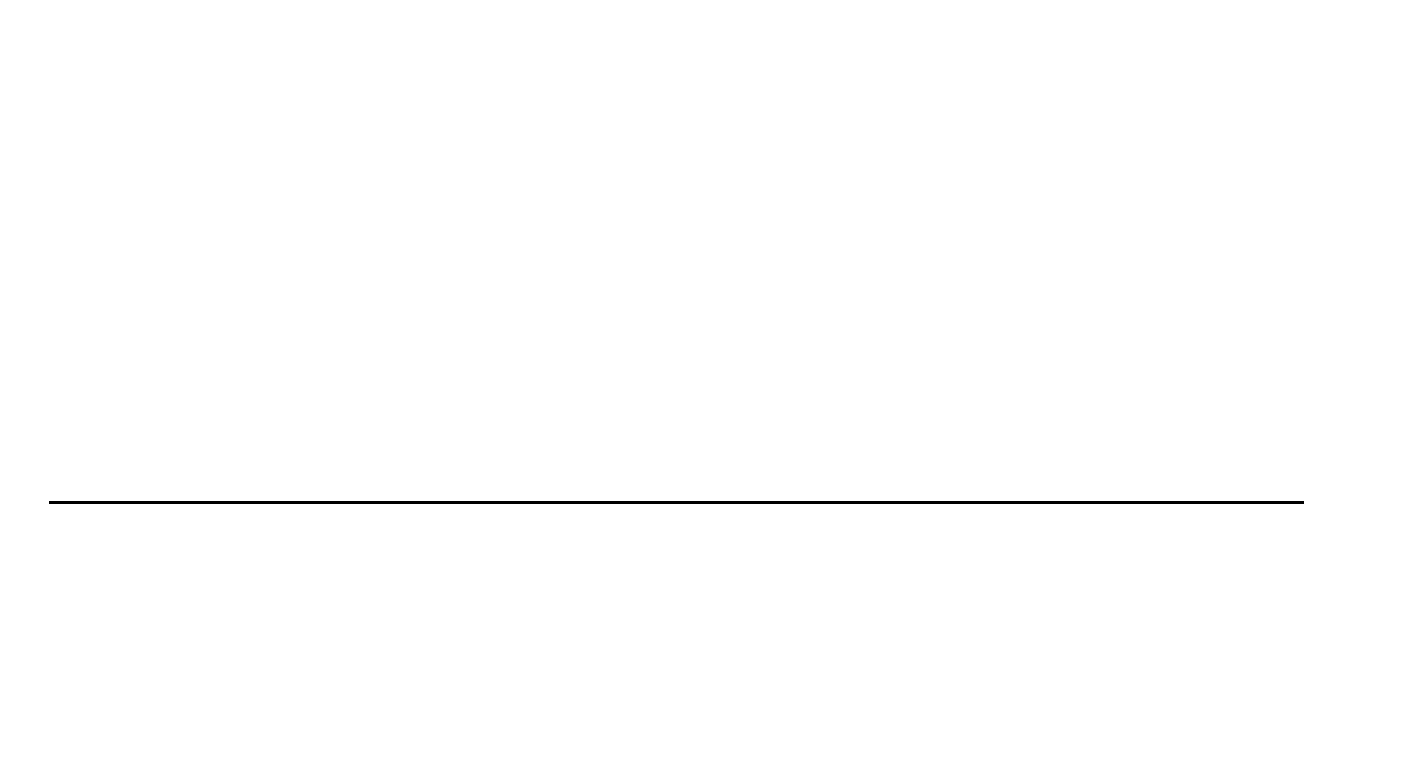
\caption{The points $(\bold{x}_{g})_i$, $(\bold{x}_{h})_i$, and $(\bold{x}_{hg})_i$ are depicted by a square, dot and star, respectively.}
\label{fig:admissibility}
\end{figure}
We turn to the computation of $\widetilde{\mu} (\bold{x}_{hg})$.

Let $(\bold{x}_{g})_i$, $(\bold{x}_{h})_i$, and $(\bold{x}_{hg})_i$ denote the components of $\bold{x}_g$, $\bold{x}_h$, and $\bold{x}_{hg}$ on $\alpha_i$, $\beta_i$, and $\gamma_i$, respectively. Suppose that $\phi\in\pi_2 (\bold{a},\bold{b},\bold{x}_{hg})$ is a Whitney triangle admitting a holomorphic representative with $n_\bold{z}(\phi) = n_\bold{w}(\phi)=0$. Let $D = \sum_i p_i D_i$ denote the domain of $\phi$. $n_\bold{w}(\phi)=0$ immediately implies that $p_1 = p_6 = 0$.

Because $(\bold{x}_{hg})_i$ is a corner of $\phi$, a standard holomorphic polygon counting argument gives that $p_5 = p_2+p_4 +1$. Suppose that $\phi$ has no corner at $(\bold{x}_g)_i$; in this case it follows that $p_5+p_7 = p_4$. Subtracting the second equation from the first yields $-p_7 = p_2 +1$, in which case either $p_7$ or $p_2$ must be negative, and $\phi$ can not admit a holomorphic representative. Thus $\phi$ has a corner at $(\bold{x}_g)_i$. A similar argument shows that $\phi$ must have a corner at $(\bold{x}_h)_i$, and that $p_5=1$ is the only nonzero multiplicity of $D$ pictured in Figure \ref{fig:admissibility}. 

Applying this argument for each local configuration as in the figure it follows that $\bold{a}=\bold{x}_g$, $\bold{b}=\bold{x}_h$, and that $D$ is a union of small triangles, in which case $\phi$ must have a unique holomorphic representative. We conclude that $\widetilde{\mu} (\bold{x}_{hg}) = \bold{x}_g\otimes \bold{x}_h$.
\end{proof}

\begin{proof}[Proof of Theorem \ref{thm:comult}]
The proof for the minus theory is similar. Because the triple diagram is weakly admissible, there is a natural multiplication chain map
\[
m^{\ge 0}: CFK^{\ge 0}(-\mathcal{H}_g)\otimes CFK^{\ge 0}(-\mathcal{H}_h) \to CFK^{\ge 0}(-\mathcal{H}_{hg}), 
\]
defined by counting pseudo-holomorphic triangles
\begin{gather*}
m^{\ge 0}([\bold{a},i_1,\dots,i_n]\otimes [\bold{b},i_1',\dots,i_n']) =\\ 
\sum\limits_{ \bold{x}\in \mathbb{T}_{\boldsymbol{\alpha}}\cap\mathbb{T}_{\boldsymbol{\gamma} }} \sum\limits_{\substack{\phi\in\pi_2 (\bold{a},\bold{b},\bold{x})\\ \mu(\phi)=0\\  n_\bold{w}(\phi) = 0}}  (\# \widehat{\mathcal{M}}(\phi)) [\bold{x},i_1+i_1'-n_{z_1}(\phi),\dots,i_n+i_n'-n_{z_n}(\phi)].
\end{gather*}
Applying the $Hom_{\mathbb{Z}_2 [U_1,\dots,U_n]}(-,\mathbb{Z}_2 [U_1,\dots,U_n,U_1 ^{-1},\dots, U_n ^{-1}]/\mathbb{Z}_2 [U_1,\dots,U_n])$ functor, and identifying $CFK^{\ge 0} (-\mathcal{H})^\vee$ with $CFK^- (\mathcal{H})$ we obtain a map
\[
\mu^-:= (m^{\ge 0})^\vee : CFK^-(\mathcal{H}_{hg})\to CFK^-(\mathcal{H}_{g})\otimes CFK^-(\mathcal{H}_{h}).
\]
The proof that $\widetilde{\mu} (\bold{x}_{hg}) = \bold{x}_g\otimes \bold{x}_h$ only used that the holomorphic triangles contributing to $\widetilde{m}$ satisfy $n_\bold{w} (\phi) = 0$; in particular, the same proof also shows that ${\mu}^- (\bold{x}_{hg}) = \bold{x}_g\otimes \bold{x}_h$.
\end{proof}

\section{the transverse invariant of a braid and its axis}
\label{sec:binding}

In this section we prove Theorem \ref{thm:binding} in three steps. First, we establish the result for the case that there is no braid and the underlying open book monodromy is simple, see Proposition \ref{prop:identity} and Lemma \ref{lemma:negative}. Second, we include the braid in the picture, see Lemma \ref{lemma:braid}. Finally, we upgrade to general open book monodromy via an application of Theorem \ref{thm:comultiplication}. 

Suppose that $(B,\pi)$ is an open book decomposition, with $B$ having $n$ components, supporting $(Y,\xi)$. Let $(S,\phi)$ be the abstract open book corresponding to $(B,\pi)$, where $S$ has genus $g$. As discussed in Subsection \ref{subsec:braidingbinding} the binding $B$ is naturally a transverse link that may be braided about the open book via a transverse isotopy and is specified by some lift $\widehat{\phi}\in Mod( S\smallsetminus \{p_1,\dots,p_n\},\partial S)$ of $\phi$, constructed by composing $\phi$ with some push-maps. See Figure \ref{fig:one} for the basis of arcs $\{a_i\}_1^{2g+n-1}\cup \{a_{2,i}\}_{2g+1}^{2g+n-1}$ for $S\smallsetminus \{p_1,\dots,p_n\}$.




The basis of arcs along with $\widehat{\phi}$ specify a Heegaard diagram $\mathcal{D}= (\Sigma, \boldsymbol{\beta}, \boldsymbol{\alpha}, \bold{w}_B, \bold{z}_B) $, along with a generator $\bold{x}_\mathcal{D}$, shown on the left side of Figure $\ref{fig:two}$ for $(-Y,B)$, as in subsection \ref{subsec:HFK}. The indexing of the basis arcs induces a labelling of the $\boldsymbol{\beta}$ and $\boldsymbol{\alpha}$ curves. The homology class $[\bold{x}_\mathcal{D}]$, in $\widehat{HFK}(\mathcal{D})$, is the braid invariant $\widehat{t}(B)$. 
Applying an isotopy to $\mathcal{D}$ we obtain the diagram depicted on the right side of Figure \ref{fig:two}, still denoted $\mathcal{D}$.

\begin{figure}[h]
\def\svgwidth{400pt}
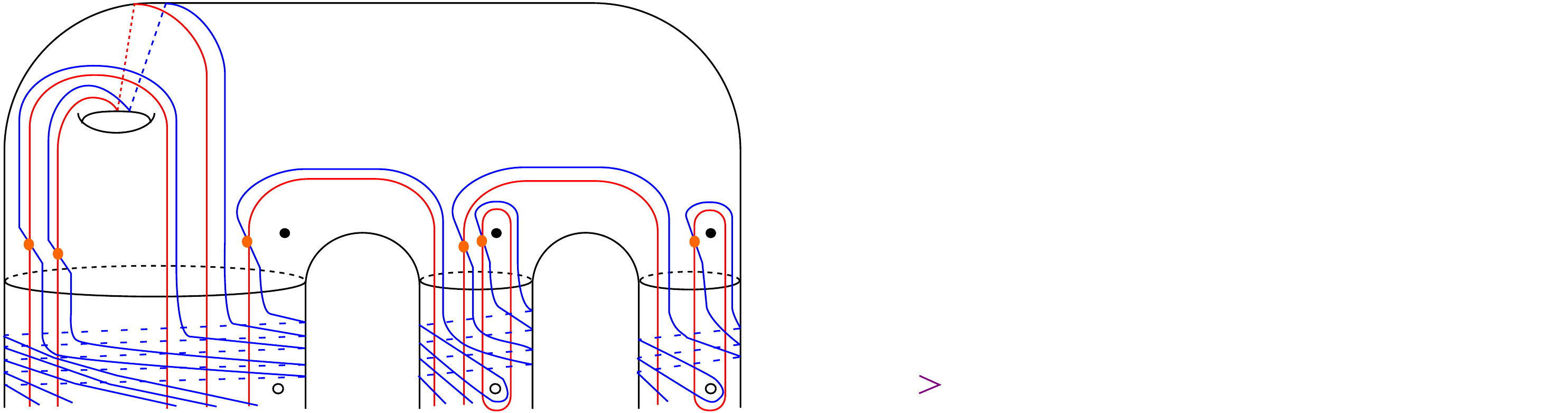
\caption{A portion of the Heegaard diagram $\mathcal{D}$ for $(-Y,B)$ in the case $g=1$ and $n=3$. The $\bold{w}_B$ basepoints and $\bold{z}_B$ basepoints are solid and hollow dots, respectively. The homology class of the generator depicted by orange dots is $\widehat{t}(B)\in \widehat{HFK}(\mathcal{D})$.}

\label{fig:two}
\end{figure}


The following notion was introduced in \cite{QOB} and is very useful for studying the relative Alexander grading.
Suppose $(\Sigma, \boldsymbol{\alpha},\boldsymbol{\beta},\bold{z}_1\cup\bold{z}_2,\bold{w}_1\cup\bold{w}_2)$ is a Heegaard diagram encoding a link $L_1\cup L_2 = L \subset Y$, where the basepoints $\bold{z}_{L_i}$ and $\bold{w}_{L_i}$ encode the sublink $L_i$.
Let $\lambda_i \subset \Sigma$ be a longitude for $L_i$ constructed by connecting the points of $\bold{z}_{L_i}$ to the $\bold{w}_{L_i}$ by oriented arcs $\{\gamma_z |z\in \bold{z}_{L_i}\}$ in the complement of $\boldsymbol{\alpha}$ curves and gluing together with arcs $\{\gamma_w |w\in \bold{w}_{L_i}\}$ connecting the $\bold{w}_{L_i}$ with the $\bold{z}_{L_i}$ in the complement of the $\boldsymbol{\beta}$ curves.
Let $D_1,\dots,D_r$ denote the closures of connected components of $\Sigma \setminus (\lambda_i\cup\boldsymbol{\alpha}\cup\boldsymbol{\beta})$. A \emph{relative periodic domain} is a 2-chain $\mathcal{P} = \Sigma a_i D_i$, whose boundary satisfies
\[
\partial \mathcal{P} = \lambda _i +\sum n_i \alpha _i + \sum m_i \beta_i.
\]

By capping off capping off the $\boldsymbol{\alpha}$ and $\boldsymbol{\beta}$ curves appearing in $\partial \mathcal{P}$ with discs one obtains an immersed surface, and hence a relative second homology class in the complement of $L_i$.


\begin{lemma} (see Lemma 2.3 of \cite{QOB})
\label{lemma:relperiodic}
Let $L_i\subset L$ be as in the definition above. Let $\mathcal{P}$ be a relative periodic domain whose homology class agrees with that of a Seifert surface $F$ for $L_i$. For $\bold{x},\bold{y}\in \mathbb{T}_{\boldsymbol{\alpha}}\cap \mathbb{T}_{\boldsymbol{\beta}}$, we have
\[
A_{L_i}(\bold{x})-A_{L_i}(\bold{y}) = n_\bold{x} (\mathcal{P})-n_\bold{y}(\mathcal{P})
\]
where the Alexander grading above is defined using the surface $F$.
\end{lemma}

In this paper we consider only the Alexander grading induced by the binding of an open book. The choice of Seifert surface will always be a page of the open book, so we continue to suppress the surface from the notation.

\begin{lemma}
\label{lemma:alex}
The Alexander grading of a generator of $\widehat{CFK}(\mathcal{D})$ is the number of its components in the region $S_{1/2}$ minus $(g+n-1)$. In particular, a generator has maximal Alexander grading ($g+n-1$) if and only if all of its components lie in the region $S_{1/2}$.
\end{lemma}
\begin{proof}
Consider the relative periodic domain $S_{1/2}$ having boundary consisting only of a longitude for $B$ (see right side of Figure $\ref{fig:two}$). Using Lemma \ref{lemma:relperiodic} it follows that a generator has all components in $S_{1/2}$ if and only if it has maximal Alexander grading. Symmetry determines the absolute grading.
\end{proof}

We first prove Theorem \ref{thm:binding} in the case that $K=\emptyset$ and the underlying monodromy of the open book is trivial, this is the content of the following technical proposition.
If $\phi$ is the identity monodromy, then $Y\simeq \# ^{2g+n-1} S^1 \times S^2$. The homology classes of oriented curves $\{A_i\}_1^{g+n-1}\cup\{B_i\}_1^{g} \subset S$ pictured in Figure $\ref{fig:four}$, freely generate $H_1(Y; \mathbb{Z})$; in particular the curves $\{A_i\}^{g+n-1}_{g+1}$ are boundary parallel.

\begin{figure}[h]
\def\svgwidth{150pt}
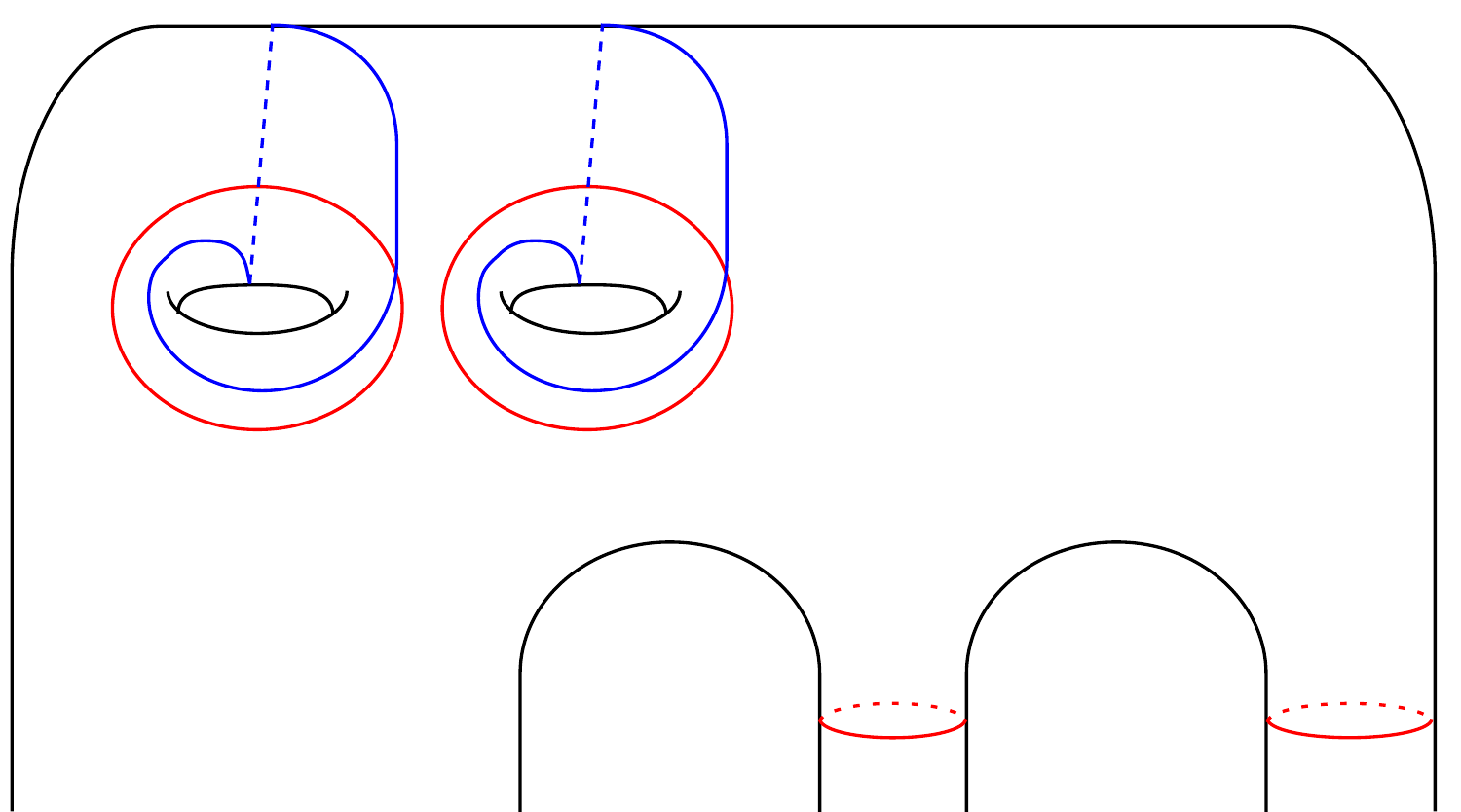
\caption{A basis for $H_1 (Y;\mathbb{Z})$, when $\phi = \bold{1}_{(S,\partial S)}$}
\label{fig:four}
\end{figure}


\begin{proposition}
\label{prop:identity}
If $\phi$ is the identity monodromy, then $\bold{x}_{\mathcal{D}}$ is the unique generator of maximal Alexander grading having $Spin^C$-grading $s_{\bold{w}_B} (\bold{x}_\mathcal{D})= s (\xi)$. In particular, $[\bold{x}_\mathcal{D}] = \widehat{t}(B) \ne 0$.
\end{proposition}

\begin{proof}

Suppose $\bold{y}\in \widehat{CFK} (\mathcal{D}, top)$ is a generator, and let $(\bold{y})^{i}$, $(\bold{y})_{i}$, denote the component of $\bold{y}$ on $\beta_{i}$, $\alpha_{i}$ respectively. 

Lemma 2.19 of \cite{holdisks3} allows us to compute the difference of $Spin^C$-gradings of generators via the Heegaard diagram.
We argue that the difference $s_{\bold{w}_B} (\bold{y})- s_{\bold{w}_B} (\bold{x}_\mathcal{D})$, a linear combination of Poincar\'{e} duals of basis elements for $H_1(Y;\mathbb{Z})$, vanishes if and only if $\bold{y}=\bold{x}_\mathcal{H}$.

We will recursively construct a sequence of generators $\{\bold{y}_j\}$ 
\[
\bold{y}=\bold{y}_{g+n},\bold{y}_{g+n-1},\dots,\bold{y}_1=\bold{x}_\mathcal{D}
\]
in $\widehat{CFK}(\mathcal{D},top)$ such that
\begin{itemize}
\item $(\bold{y}_j)^{i}=(\bold{x}_\mathcal{D})^{i}$ for $j>g$ and $i\ge g+j$,
\item $(\bold{y}_j)^{i}=(\bold{x}_\mathcal{D})^{i}$ and $(\bold{y}_j)^{2,k}=(\bold{x}_\mathcal{D})^{2,k}$ for $j\le g$, $i\ge 2j-1$, and any $k$.
\item $s_{\bold{w}_B} (\bold{y}_{j+1}) -s_{\bold{w}_B} (\bold{y}_{j})$ is a linear combination of $PD([A_j])$ and $PD([B_j])$ (the latter only appears for $j\le g$). The difference is equal to zero if and only if $\bold{y}_{j+1} = \bold{y}_{j}$.
\end{itemize}

\begin{enumerate}[leftmargin=*]

\item
We begin by constructing $\bold{y}_{g+n-1}$. If $(\bold{y})^{2g+n-1} = (\bold{x}_\mathcal{D})^{2g+n-1}$ then we set $\bold{y}_{g+n-1} = \bold{y}$, so assume $(\bold{y})^{2g+n-1} \ne (\bold{x}_\mathcal{D})^{2g+n-1}$. There are two other intersection points involving $\beta_{2g+n-1}$ in the region $S_{1/2}$, one with $\alpha_{2g+n-1}$ and another with $\alpha_{2,2g+n-1}$.\\

\begin{enumerate}
\item
If $(\bold{y})^{2g+n-1}=(\bold{y})_{2g+n-1}$ let $\bold{y}_{g+n-1}$ be obtained from $\bold{y}$ by replacing $(\bold{y})^{2g+n-1}$ with $(\bold{x}_\mathcal{D})^{2g+n-1}$. Then $s_{\bold{w}_B} (\bold{y}) -s_{\bold{w}_B} (\bold{y}_{g+n-1}) = PD([A_{g+n-1}])$. \\

\item
If $(\bold{y})^{2g+n-1}=(\bold{y})_{2,2g+n-1}$ then it must be the case that $(\bold{y})^{2,2g+n-1}=(\bold{y})_{2g+n-1}$. Let $\bold{y}_{g+n-1}$ be obtained from $\bold{y}$ by replacing $(\bold{y})^{2g+n-1}$, $(\bold{y})^{2,2g+n-1}$, with $(\bold{x}_\mathcal{D})^{2g+n-1}$, $(\bold{x}_\mathcal{D})^{2,2g+n-1}$, respectively. Then $s_{\bold{w}_B} (\bold{y}) -s_{\bold{w}_B} (\bold{y}_{g+n-1}) = PD([A_{g+n-1}])$.\\
\end{enumerate}

\begin{figure}[h]
\def\svgwidth{180pt}
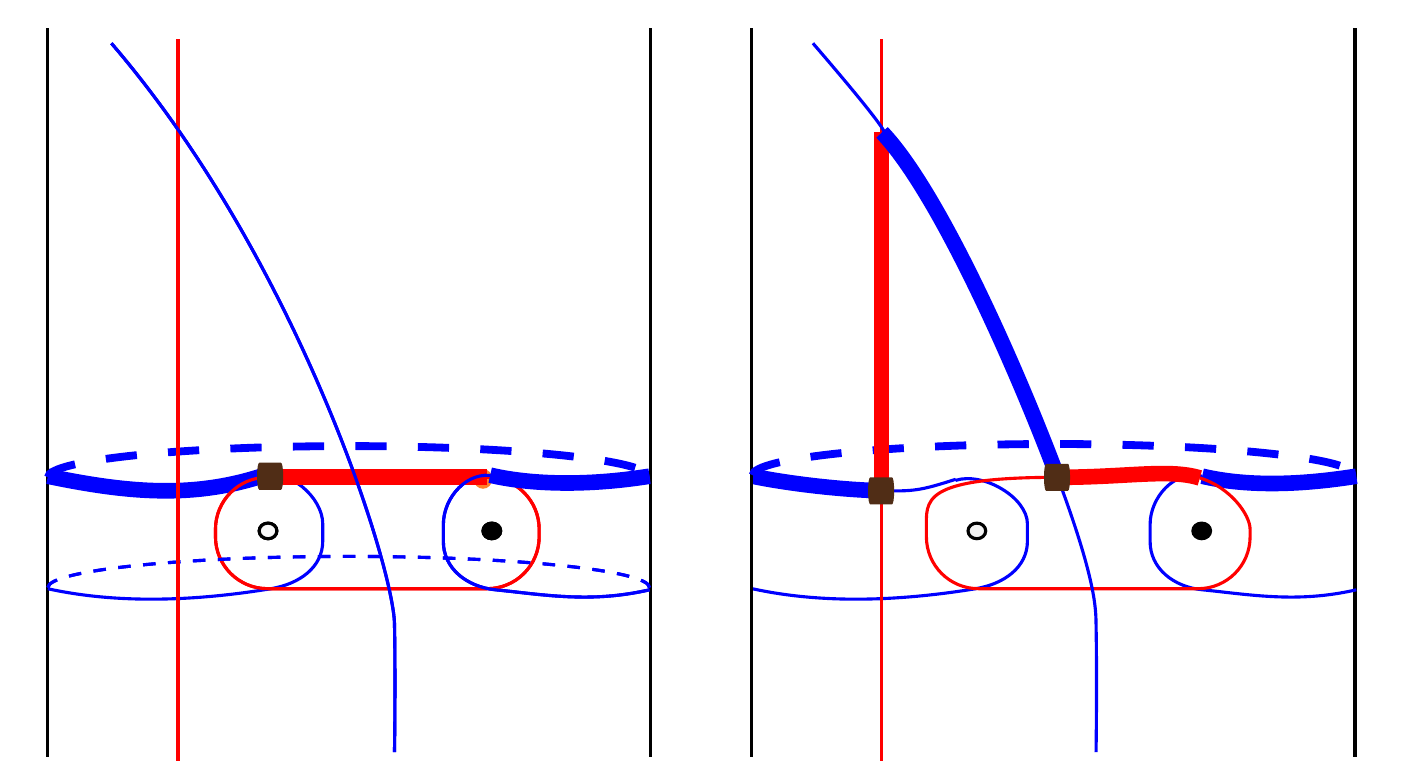
\caption{\textbf{Step 1}: The construction of $\bold{y}_{g+n-1}$, both cases. The orange dots are components of $\bold{x}_\mathcal{H}$ and $\bold{y}_{g+n-1}$ (they agree here by construction). The brown squares represent $\bold{y}$. The highlighted path $\epsilon (\bold{y},\bold{y}_{g+n-1})$ is homologous to $[A_{g+n-1}]$ in both cases.}
\label{fig:five}
\end{figure}

\item
Assuming that we have constructed $\bold{y}_{j+1}$ for some $j>g$ satisfying the hypotheses, we construct $\bold{y}_j$. 
If $(\bold{y}_{j+1})^{g+j} = (\bold{x}_\mathcal{D})^{g+j}$ then we set $\bold{y}_{j} = \bold{y}_{j+1}$, so assume $(\bold{y}_{j+1})^{g+j} \ne (\bold{x}_\mathcal{D})^{g+j}$. There are three other intersection points involving $\beta_{g+j}$ in the region $S_{1/2}$, one with $\alpha_{g+j}$, another with $\alpha_{2,g+j}$ and a third with $\alpha_{2,g+j+1}$. \\

\begin{enumerate}
\item
If $(\bold{y}_{j+1})^{g+j} = (\bold{y}_{j+1})_{g+j}$, let $\bold{y}_j$ be obtained from $\bold{y}_{j+1}$ by replacing $(\bold{y}_{j+1})^{g+j}$ with $(\bold{x}_\mathcal{D})^{g+j}$. As in the first case of the construction of $\bold{y}_{g+n-1}$ it follows that $s_{\bold{w}_B} (\bold{y}_{j+1}) -s_{\bold{w}_B} (\bold{y}_{j}) = PD([A_{j}])$.\\

\item
If $(\bold{y}_{j+1})^{g+j} = (\bold{y}_{j+1})_{2,g+j}$ there are two further sub-cases. Either $(\bold{y}_{j+1})_{g+j} = (\bold{y}_{j+1})^{2,g+j+1}$ and $(\bold{y}_{j+1})^{2,g+j} = (\bold{y}_{j+1})_{2,g+j+1}$,
 or $(\bold{y}_{j+1})_{g+j} = (\bold{y}_{j+1})^{2,g+j}$. \\

\begin{enumerate}
\item
 In the former case, let $\bold{y}_j$ be obtained from $\bold{y}_{j+1}$ by replacing $(\bold{y}_{j+1})^{g+j}$ with $(\bold{x}_{\mathcal{D}})^{g+j}$, $(\bold{y}_{j+1})^{2,g+j}$ with $(\bold{x}_{\mathcal{D}})^{2,g+j}$, and $(\bold{y}_{j+1})^{2,g+j+1}$ with $(\bold{x}_{\mathcal{D}})^{2,g+j+1}$.\\

\item
In the latter case, let $\bold{y}_j$ be obtained from $\bold{y}_{j+1}$ by replacing $(\bold{y}_{j+1})^{g+j}$ with $(\bold{x}_{\mathcal{D}})^{g+j}$ and $(\bold{y}_{j+1})^{2,g+j}$ with $(\bold{x}_{\mathcal{D}})^{2,g+j}$. In either case we have that $s_{\bold{w}_B} (\bold{y}_{j+1}) -s_{\bold{w}_B} (\bold{y}_{j}) = PD([A_{j}])$. See Figure \ref{fig:six}.\\
\end{enumerate}
 
 \begin{figure}[h]
\def\svgwidth{195pt}
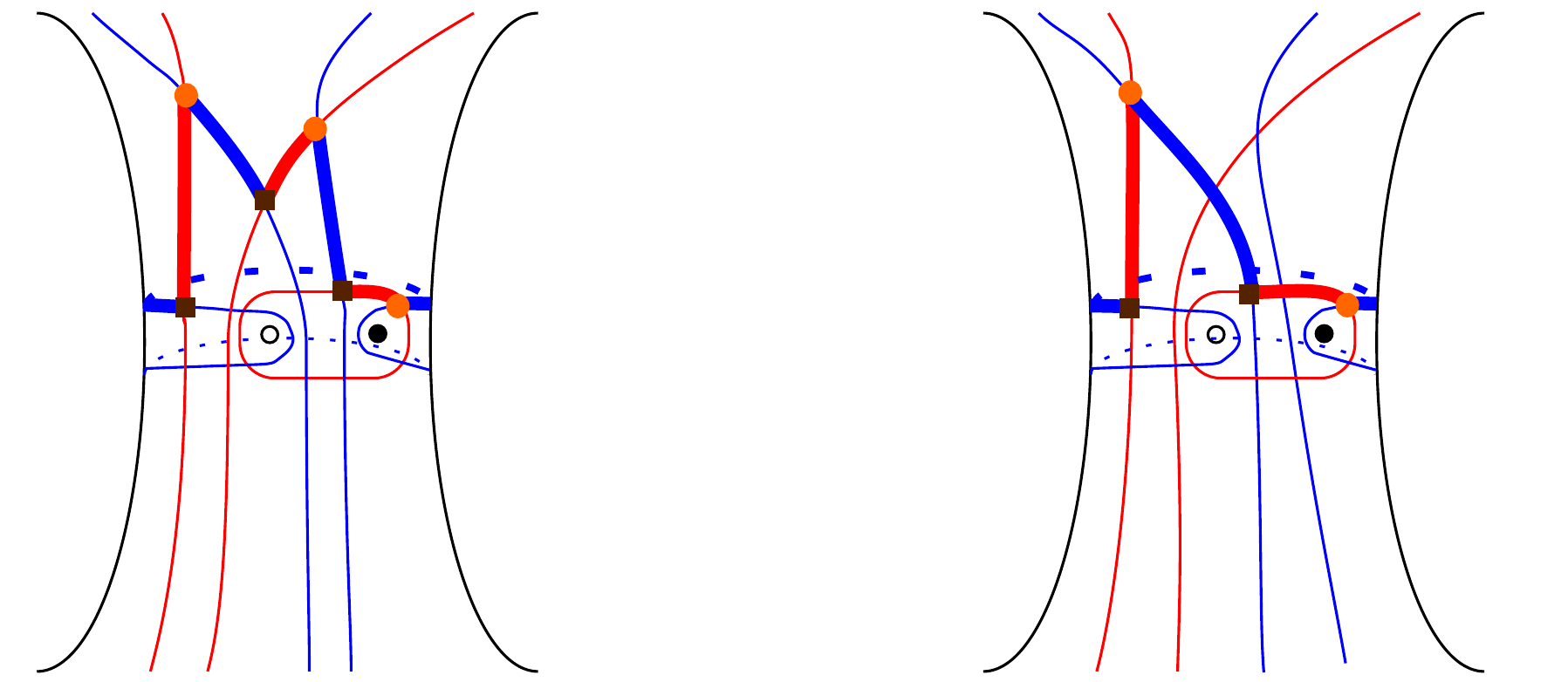
\caption{\textbf{Case 2b}: The construction of $\bold{y}_{j}$, in both sub-cases. The orange dots are components of $\bold{x}_\mathcal{H}$ and $\bold{y}_{j}$ (they agree here by construction). The brown squares are components of $\bold{y}_{j+1}$. The highlighted path $\epsilon (\bold{y}_{j+1},\bold{y}_{j})$ is homologous to $[A_{j}]$ in both cases.}
\label{fig:six}
\end{figure}
 
\item
If $(\bold{y}_{j+1})^{g+j} = (\bold{y}_{j+1})_{2,g+j+1}$ then it follows that $(\bold{y}_{j+1})_{g+j} = (\bold{y}_{j+1})^{2,g+j+1}$. Let $\bold{y}_{j}$ be obtained from $\bold{y}_{j+1}$ by replacing $(\bold{y}_{j+1})^{g+j}$, $(\bold{y}_{j+1})^{2,g+j+1}$, with $(\bold{x}_\mathcal{D})^{g+j}$, $(\bold{x}_\mathcal{D})^{2,g+j+1}$, respectively. The picture is similar to that of the second case (1b) of the construction of $\bold{y}_{g+n-1}$.
It follows that $s_{\bold{w}_B} (\bold{y}_{j+1}) -s_{\bold{w}_B} (\bold{y}_{j}) = PD([A_{j}])$.
\end{enumerate}

\vspace{.5cm}
\newpage
\item

Assuming that we have constructed $\bold{y}_{j+1}$ for some $j\le g$ satisfying the hypotheses, we construct $\bold{y}_j$.  For each $i > 2g$, in the domain $S_{1/2}$ the curve $\beta_{2,i}$ intersects only $\alpha_{2,i}$ (among the $\alpha$-curves whose components of $\bold{y}_{j+1}$ are not fixed); i.e. $(\bold{y}_{j+1})^{2,i} = (\bold{y}_{j+1})_{2,i}$. Of the remaining $\alpha$-curves, $\beta _{2j}$ and $\beta_{2j-1}$ intersect only $\alpha_{2j}$ and $\alpha_{2j-1}$. There are two cases, with three nontrivial sub-cases each. Either $(\bold{y}_{j+1})^{2j}=(\bold{y}_{j+1})_{2j}$ and $(\bold{y}_{j+1})^{2j-1}=(\bold{y}_{j+1})_{2j-1}$, or $(\bold{y}_{j+1})^{2j}=(\bold{y}_{j+1})_{2j-1}$ and $(\bold{y}_{j+1})^{2j-1}=(\bold{y}_{j+1})_{2j}$.
The trivial sub-case is that $(\bold{y}_{j+1})^{2j}=(\bold{x}_\mathcal{D})^{2j}$ and  $(\bold{y}_{j+1})^{2j-1}=(\bold{x}_\mathcal{D})^{2j-1}$, where we set $\bold{y}_j = \bold{y}_{j+1}$. \\
\begin{enumerate}
\item
Suppose that $(\bold{y}_{j+1})^{2j}=(\bold{y}_{j+1})_{2j}$ and $(\bold{y}_{j+1})^{2j-1}=(\bold{y}_{j+1})_{2j-1}$.\\
\begin{enumerate}
\item

If $(\bold{y}_{j+1})^{2j}\ne (\bold{x}_\mathcal{D})^{2j}$ and $(\bold{y}_{j+1})^{2j-1}=(\bold{x}_\mathcal{D})^{2j-1}$ then, obtaining $\bold{y}_j$ from $\bold{y}_{j+1}$ by replacing $(\bold{y}_{j+1})^{2j}$ with  $(\bold{x}_\mathcal{D})^{2j}$, we have that $s_{\bold{w}_B} (\bold{y}_{j+1}) -s_{\bold{w}_B} (\bold{y}_{j}) =PD([B_j])$. See the left side of Figure \ref{fig:seven}.\\
\item
If $(\bold{y}_{j+1})^{2j}= (\bold{x}_\mathcal{D})^{2j}$ and $(\bold{y}_{j+1})^{2j-1}\ne(\bold{x}_\mathcal{D})^{2j-1}$ then, obtaining $\bold{y}_j$ from $\bold{y}_{j+1}$ by replacing $(\bold{y}_{j+1})^{2j-1}$ with  $(\bold{x}_\mathcal{D})^{2j-1}$, we have that $s_{\bold{w}_B} (\bold{y}_{j+1}) -s_{\bold{w}_B} (\bold{y}_{j}) =PD([A_j])$. See the right side of Figure \ref{fig:seven}\\
\item
If $(\bold{y}_{j+1})^{2j}\ne (\bold{x}_\mathcal{D})^{2j}$ and $(\bold{y}_{j+1})^{2j-1}\ne(\bold{x}_\mathcal{D})^{2j-1}$ then, we obtain $\bold{y}_j$ from $\bold{y}_{j+1}$ by replacing $(\bold{y}_{j+1})^{2j}$ with  $(\bold{x}_\mathcal{D})^{2j}$ and $(\bold{y}_{j+1})^{2j-1}$ with  $(\bold{x}_\mathcal{D})^{2j-1}$. Taking the union of the paths of the two previous sub-cases, we have that  $s_{ \bold{w}_B} (\bold{y}_{j+1}) -s_{ \bold{w}_B} (\bold{y}_{j}) =PD([A_j]+[B_j])$. 
\end{enumerate}


 \begin{figure}[h]
\def\svgwidth{185pt}
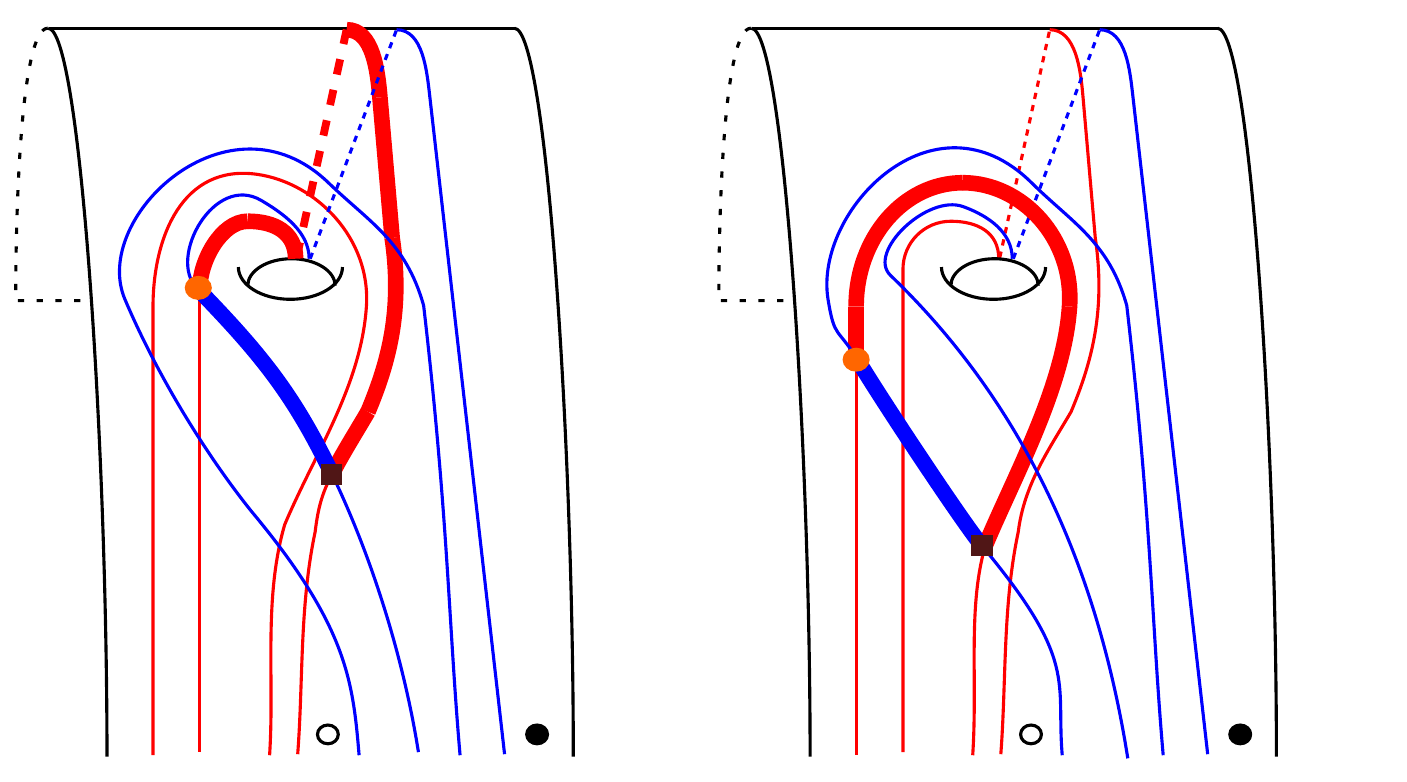
\caption{\textbf{Case 3a}: The construction of $\bold{y}_{j}$, in the first and second sub-cases. The orange dots are components of $\bold{x}_\mathcal{H}$ and $\bold{y}_{j}$. The brown squares are components of $\bold{y}_{j+1}$. The highlighted path $\epsilon (\bold{y}_{j+1},\bold{y}_{j})$ is homologous to $[B_j]$ in the first sub-case, and $[A_{j}]$ in the second.}
\label{fig:seven}
\end{figure}

\item
Now, assume we are in the case that $(\bold{y}_{j+1})^{2j}=(\bold{y}_{j+1})_{2j-1}$ and $(\bold{y}_{j+1})^{2j-1}=(\bold{y}_{j+1})_{2j}$. The intersection point $(\bold{y}_{j+1})^{2j}$ is fixed, and there are three possibilities for $(\bold{y}_{j+1})^{2j-1}$. In each of these sub-cases we obtain $\bold{y}_j$ from $\bold{y}_{j+1}$ by replacing $(\bold{y}_{j+1})^{2j}$ with  $(\bold{x}_\mathcal{D})^{2j}$ and $(\bold{y}_{j+1})^{2j-1}$ with  $(\bold{x}_\mathcal{D})^{2j-1}$. Figure \ref{fig:eight} shows the three possible sub-cases; in each case $s_{ \bold{w}_B } (\bold{y}_{j+1}) -s_{\bold{w}_B} (\bold{y}_{j})$ is a non-trivial linear combination of $PD([A_j])$ and $PD([B_j])$. 
\end{enumerate}
\end{enumerate}

 \begin{figure}[h]
\def\svgwidth{195pt}
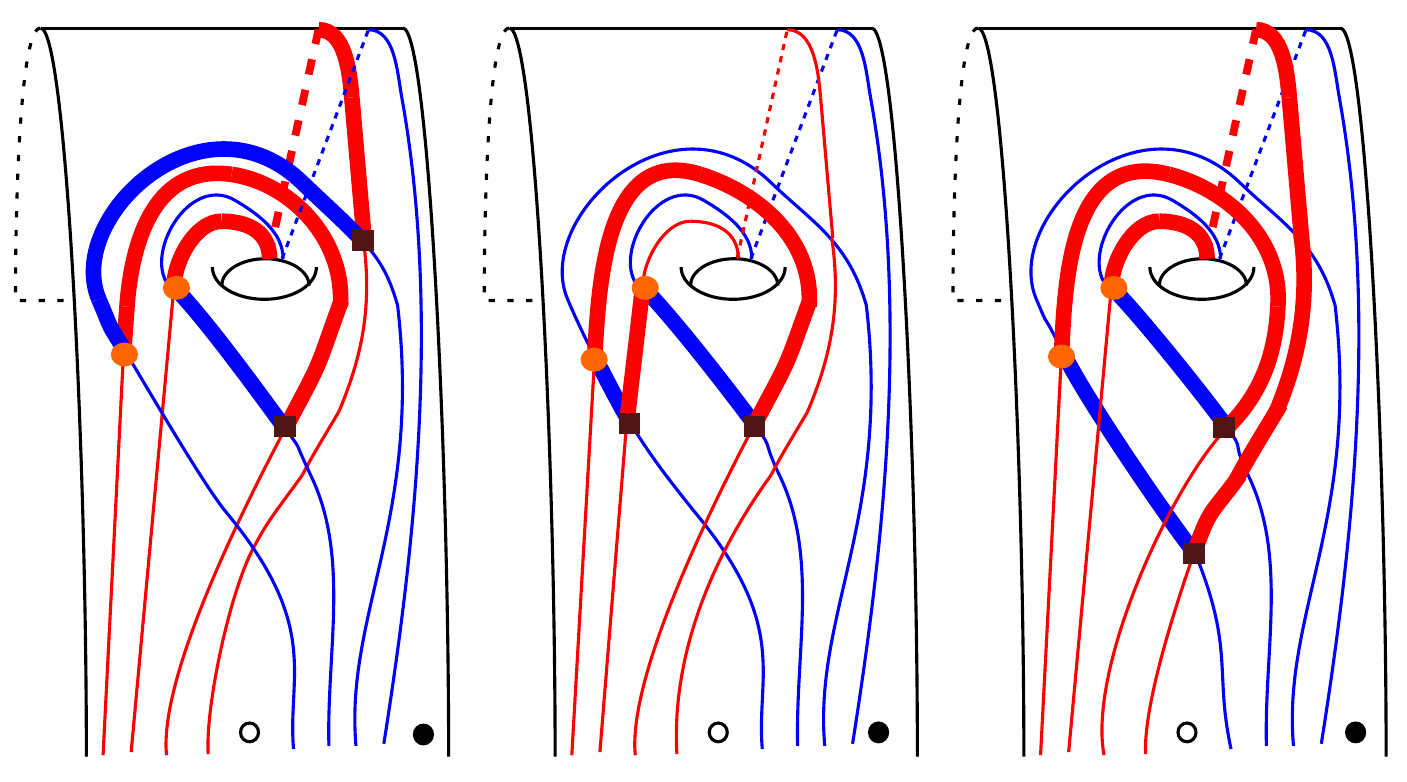
\caption{\textbf{Case 3b}: The construction of $\bold{y}_{j}$, in the case that $(\bold{y}_{j+1})^{2j}=(\bold{y}_{j+1})_{2j-1}$. The orange dots are components of $\bold{x}_\mathcal{H}$ and $\bold{y}_{j}$. The brown squares are components of $\bold{y}_{j+1}$. The highlighted path $\epsilon (\bold{y}_{j+1},\bold{y}_{j})$ is homologous to $[B_j]$ in the first sub-case, $[A_{j}]$ in the second, and $[A_j]+[B_j]$ in the third.}
\label{fig:eight}
\end{figure}

\end{proof}


Still assuming $K=\emptyset$, we now consider a monodromy $\phi$ specifying a 3-manifold whose first homology is compatible with the calculation of Proposition \ref{prop:identity}.

\begin{lemma}
\label{lemma:negative}
Suppose an open book $(B,\pi)$ has monodromy
\[\phi = P\circ \displaystyle\prod_{C \in \Delta} \tau_{C}^{m},\] where $P$ is a product of right and left-handed Dehn twists about separating curves in $S$, $\Delta$ is the set of $n$ connected components of $\partial S$, and $m\in \mathbb{Z}$. Then $\bold{x}_{\mathcal{D}}$ is the unique generator of maximal Alexander grading having $Spin^C$-grading $s_{\bold{w}_B} (\bold{x}_\mathcal{D})= s (\xi)$. In particular, $[\bold{x}_\mathcal{D}] = \widehat{t}(B) \ne 0$.

\end{lemma}
\begin{proof}

Given the monodromy $\phi$, one readily obtains a presentation for $H_1 (Y)$ as in Subsection 2.1 of \cite{monhom}. 
We use the generators $\{A_i\}_1^{g+n-1}\cup \{B_i\}_1^g$ described above Figure \ref{fig:four}.
Dehn twists about separating curves do not introduce any relations.
The composition of boundary parallel Dehn twists $\displaystyle\prod_{C \in D} \tau_{C}^{m}$ introduces the $n-1$ relations \[\{nm(A_i) =0 | i>g\}.\]
If $n=1$, there are no new relations and the $Spin^C$-structure calculation of Proposition \ref{prop:identity} carries through. If $n>1$,
consider the group homomorphism $h: H_1(Y)\to \mathbb{Z}^{2g} \oplus \mathbb{Z}/n\mathbb{Z}$ specified by

\begin{itemize}
\item $h(A_i) = e_i$ for $i\le g$
\item $h(B_i) = e_{g+i}$ 
\item $h(A_i) = 1\in\mathbb{Z}/n\mathbb{Z}$ for $i>g$.
\end{itemize}
where $\{e_1,\dots,e_{2g}\}$ is a basis for $\mathbb{Z}^{2g}$ and $1\in\mathbb{Z}/n\mathbb{Z}$ is a generator. Note that for any subset \[R\subset \{g+1,g+2,\dots, g+n-1\} \text{ we have } \sum\limits_{i\in R} A_i \not\in ker(h).\]  
In particular $\sum\limits_{i\in R} A_i \in H_1(Y)$ is a nontrivial homology class. The $Spin^C$-structure calculation of Proposition \ref{prop:identity} carries through and shows that $\bold{x}_{\mathcal{D}}$ is the unique generator of maximal Alexander grading having $Spin^C$-grading $s_{\bold{w}_B} (\bold{x}_\mathcal{D})= s (\xi)$. 

\end{proof}

Suppose $K$ is a link braided about $(B,\pi)$ of braid index $k$. 
Consider a basis of arcs for $S\smallsetminus \{c_1,\dots, c_k, p_1,\dots,p_n\}$ obtained from $\{a_i\}_1 ^{2g+n-1} \cup \{a_{2,i}\}_{2g+1}^{2g+n-1}$ by adding $k$ boundary parallel arcs $\{a_{3,i}\}_{1}^{k}$ each of which (together with a sub-arc of $\partial S$) bound a disk containing one of the points $c_i$. 
This basis, together with a pointed monodromy having transverse closure $B\cup K$ (as in Remark \ref{remark:a}), specifies a Heegaard diagram $\mathcal{H} = (\Sigma, \boldsymbol{\beta},\boldsymbol{\alpha}, \bold{w}_K\cup \bold{w}_B,\bold{z}_K\cup \bold{z}_B)$, see Figure \ref{fig:nine}, which encodes $(-Y,B\cup K)$. By construction, $[\bold{x}_\mathcal{H}] \in \widehat{HFK}(\mathcal{H})$ is $\widehat{t}(B\cup K)$.

 \begin{figure}[h]
\def\svgwidth{170pt}
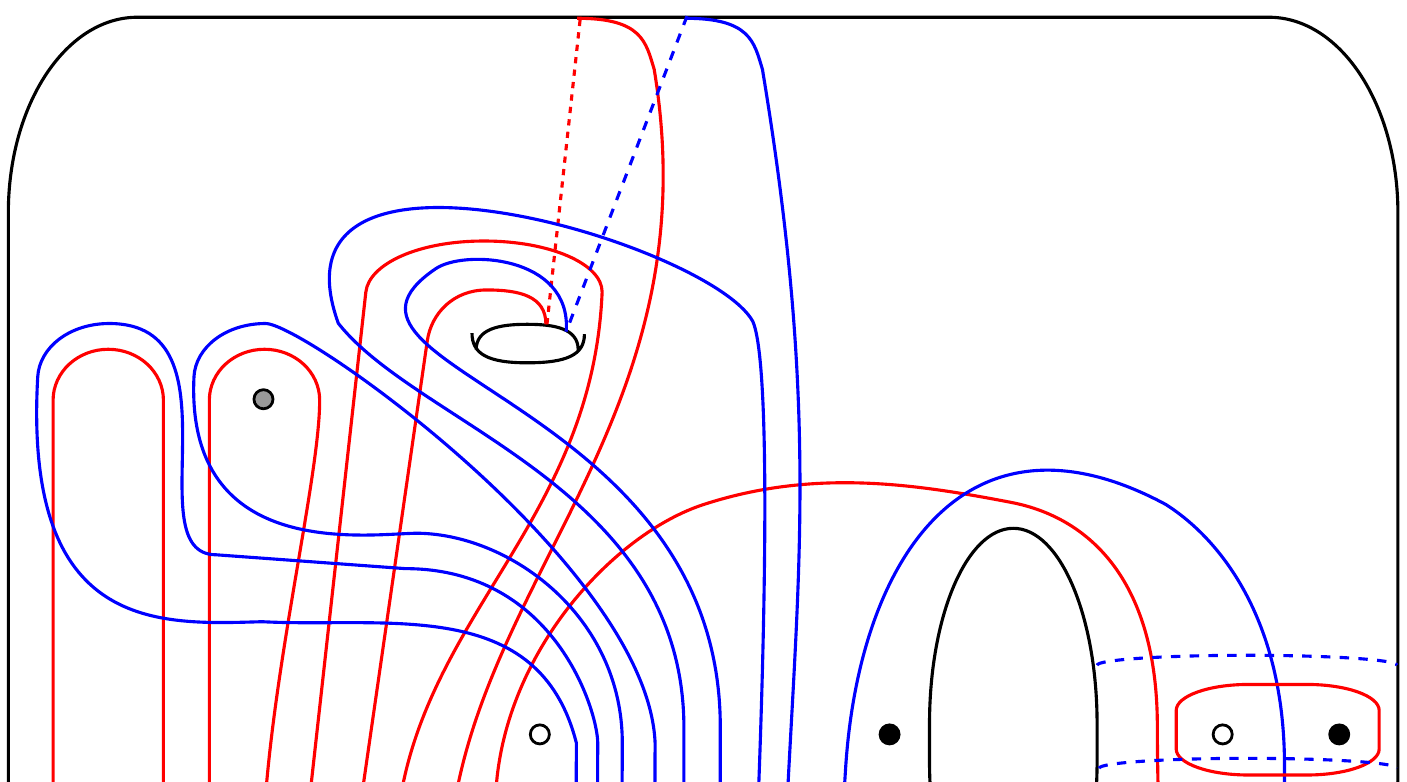
\caption{A portion of the diagram $\mathcal{H}$ in the case $k=2$, $g=1$, and $n=2$. The orange dots are components of $\bold{x}_\mathcal{H}$. Points of $\bold{w}_K$ are depicted with gray basepoints. The two left-most $\boldsymbol{\alpha}$-curves are $\alpha_{3,1}$ and $\alpha_{3,2}$.}
\label{fig:nine}
\end{figure}


\begin{lemma}
\label{lemma:braid}
Suppose $K$ is braided about an open book $(B,\pi)$ with open book monodromy
\[\phi = P\circ \displaystyle\prod_{C \in \Delta} \tau_{C}^{m},\] where $P$ is a product of right and left-handed Dehn twists about separating curves in $S$, $\Delta$ is the set of $n$ connected components of $\partial S$, and $m\in \mathbb{Z}$, then $\widehat{t}(B\cup K)\ne 0$.

\end{lemma}

\begin{proof}
The binding $B$ induces an Alexander grading on $\widehat{CFK}(\mathcal{H})$;
for a generator  $\bold{y}\in \mathbb{T}_{\boldsymbol{\beta}} \cap \mathbb{T}_{\boldsymbol{\alpha}}$ let $A_B (\bold{y})$ denote this grading.
As in Lemma \ref{lemma:alex} it is easy to argue that a generator $\bold{y}$ has maximal $A_B$ grading if and only if all components of $y$ lie in $S_{1/2}$.

Observe that in $S_{1/2}$ for any $i\le k$ the curve $\alpha_{3,i}$ only intersects curves $\beta_{3,j}$ for $j\le i$, and no other $\beta$ curves. Thus if $A_B (\bold{y})$ is maximal, then $(\bold{y})_{3,i }=(\bold{y})^{3,i }$ for $1\le i\le k$. Let $C$ denote the subcomplex of $\widehat{CFK}(\mathcal{H})$ in maximal $A_B$ grading. The underlying module of $C$ can be identified with $(V_1\otimes V_2 \otimes \dots \otimes V_k)\otimes \widehat{CFK}(\mathcal{D}, top)$, where each $V_i$ is a rank two $\mathbb{Z}_2$-module with generators $x_i$ and $y_i$, thought of as the points of $\boldsymbol{\beta}_{3,i}\cap \boldsymbol{\alpha}_{3,i}$ where $x_i = (\bold{x}_\mathcal{H})^{3,i}$.
By Lemma \ref{lemma:negative}, the summand $S$ of $C$ whose $Spin^C$ grading agrees with $s_{\bold{w}_B}(\bold{x}_{\mathcal{H}})$ is then identified, as a module, with $(V_1\otimes V_2 \otimes \dots \otimes V_k)\otimes \bold{x}_{\mathcal{D}}$. $S$ further decomposes as a direct sum over the relative $Spin^C$ structures on the complement of $K$; each generator of $S$ lies in a unique such grading so the differential on $S$ is trivial. In particular, $0\ne [x_1\otimes\dots\otimes x_k \otimes \bold{x}_\mathcal{D}] = [\bold{x}_\mathcal{H}] = \widehat{t}(B\cup K)$.

\end{proof}

We now consider general monodromy $\phi$, and prove Theorem \ref{thm:binding}.
\begin{proof}
Let $S$ denote a compact and oriented surface with boundary. Let $C\subset S$ denote a non-separating, non-boundary-parallel, simple closed curve, and $C'$ another such curve intersecting $C$ transversely in one point.
By applying a sequence of conjugations and compositions with left-handed Dehn twists about non-separating curves, one may pass from $\tau_C^{-1}$ to $(\tau_C \circ \tau_{C'})^{-6}$; by the chain relation (see Proposition 4.12 of \cite{primer}), the latter is equivalent to a left-handed Dehn twist about a separating curve, the boundary of a regular neighborhood of $C\cup C'$ in $S$. 

Any monodromy $\phi\in Mod(S,\partial S)$ can be expressed as a product of right and left-handed Dehn twists.
By applying a sequence of conjugations and compositions with left-handed Dehn twists about non-separating curves, one can pass from $\phi$ to an open book monodromy as in Lemma \ref{lemma:braid}; combining this with Proposition \ref{prop:nonzero} and Theorem \ref{thm:comultiplication} completes the proof.
\end{proof}

\section{Proof of Theorem 1.3}

In this section we prove Theorem \ref{thm:largetwist} by combining the following two Lemmas with Theorems \ref{thm:binding} and \ref{thm:comultiplication} and Proposition \ref{prop:nonzero}.

\begin{lemma}
\label{lem:simple}
Suppose that $c\subset S$ is an embedded arc connecting $\partial S$ to $p\in P$, disjoint from all other points of $P$. Let $\phi \in Mod(S\smallsetminus P,\partial S)$ be a monodromy such that $\phi(c)$ is to the right of $c$, and disjoint from $c$ along its interior.

If $\phi$ fixes $p$ there exists a simple closed curved $\delta\in S\smallsetminus P$ such that the right-hand Dehn twist $\tau_\delta$ maps $c$ to $\phi(c)$.
If $\phi$ does not fix $p$ there exists an embedded arc $\gamma$ from $p$ to $\phi(p)$ such that the right-hand half twist $\sigma_{\gamma}$ maps $c$ to $\phi(c)$.
\end{lemma}

\begin{proof}
Suppose that $\phi$ fixes $p$. Let $\alpha = c \cup \phi (c)$. We apply an isotopy supported in a neighborhood of $\alpha\cap \partial S$ to push $\alpha$ into the interior of $S$, still denoting the resulting curve $\alpha$. The curve $\alpha$ admits an annular neighborhood $S^1\times I\simeq N(\alpha)\subset S$, so that $N(\alpha)\cap P = p$. $N(\alpha)$ has two boundary components, one which is disjoint from $c$, and one intersecting $c$ transversely in one point. Letting $\delta$ denote the latter component of $\partial N(\alpha)$, clearly $\tau_\delta (c) = \phi(c)$.

If $\phi$ does not fix $p$ let $\gamma = c\cup \phi (c)$. We apply a small isotopy supported in a neighborhood of $\gamma\cap \partial S$ to push $\gamma$ into the interior of $S$. It is clear that $\sigma_\gamma (c) = \phi (c)$.
\end{proof}

\begin{remark}
The following Lemma is similar to Lemma 5.2 of \cite{RVD}, where a surface with possibly multiple boundary components, but no marked points, is considered. 


\end{remark}

We denote, as in Subsection \ref{subsec:FDTC}, that $\phi (a)$ is to the right of $a$ by $a\le \phi(a)$.

\begin{lemma}
\label{lem:sequence}
Suppose that $e_0\subset S$ is an embedded arc connecting $\partial S$ to $p\in P$, disjoint from all other marked points. Suppose $\phi \in Mod(S\smallsetminus P,\partial S)$ takes the arc $e_0$ to the right, i.e. $e_0\le \phi(e_0)$. Then there exists a sequence of arcs
\[
e_0\le e_1\le e_2\le \dots\le e_m = \phi (e_0)
\]
connecting $e_0\cap \partial S$ to a point of $P$, such that $e_i$ and $e_{i+1}$ are disjoint along their interiors for each $0\le i\le m$.
\end{lemma}
\begin{proof}
To simplify notation let $a := e_0$ and $c:= \phi (e_0)$.
We let $\#(a,c)$ denote the geometric intersection number of $a$ with the interior of $c$.
 It suffices to show that if $a\le c$, $a\ne c$, and $\#(a,c)\ne 0$ then there exists a properly embedded arc $b$ such that $a\le b\le c$ and $\#(a,b),\#(b,c)< \#(a,c)$. The desired sequence of arcs can be obtained by iterating the construction.

We endow $S\smallsetminus P$ with a hyperbolic metric and assume that $a$ and $c$ are parametrized geodesic arcs intersecting in a collection of points $\{x_1,\dots,x_m\}=\{y_1,\dots,y_m\}$ where
\begin{gather*}
x_i = a(t_i) \text{ for some } 0=t_0\le t_1\le \dots \le t_m = 1\\
y_i = c(s_i) \text{ for some } 0=s_0\le s_1\le \dots \le s_m = 1.
\end{gather*}

The proof is a case by case analysis.

\begin{enumerate}[leftmargin=*]
\item
Suppose that $x_1=y_r$, and that at $x_1$ the tangent vectors to $a$ and $c$, in this order, form a negative basis for $T_{x_1}(S)$. Consider $b = a|_{[0,t_1]}\ast c|_{[s_r,1]}$, by which we mean the concatenation. Clearly we have that $\#(a,b) = \#(a,c)-r$ and $\#(b,c) = 0$. Note that $b$ is veers to the left of $c$ near $\partial S$, since they intersect efficiently it follows that $b\le c$. When smoothed, the piecewise geodesic path $b$ veers to the right of $a$.
\item

Suppose now that $x_1 = y_r$, with $r>1$, and that at $x_1$ the tangent vectors to $a$ and $c$ form a positive basis for $T_{x_1}(S)$. Let $x_{r'}= y_{r''}$ be the last point of $a$ to intersect $c|_{[0,s_r]}$.  There are two sub-cases:
\begin{enumerate}
\item

Suppose that the tangent vectors to $a$ and $c$ form a negative basis for $T_{x_{r'}}(S)$. Set $b = c|_{[0,s_{r''}]}\ast a|_{[t_{r'},1]}$. It is clear that $\#(a,b),\#(b,c)<\#(a,c)$. 

We claim that $a\le b\le c$. 
We pass to the universal cover $\pi: \widetilde{S}\to S$. Let $\widetilde{a}$ and $\widetilde{c}$ be lifts of $a$ and $c$ starting at some fixed lift $\widetilde{x_0}$ of $x_0$. There is a natural lift $\widetilde{b}$ of $b$ which consists of geodesic arcs along $\widetilde{c}$ and some lift of $a$ different from $\widetilde{a}$. In particular, $\widetilde{b}$ starts along $\widetilde{c}$ to the right of $\widetilde{a}$, and then switches to some other lift of $a$. Because any two distinct lifts of $a$ are disjoint, it follows that $a\le b$. Since $\widetilde{b}$ starts along $\widetilde{c}$ and diverges to the left along a lift of $a$, it is clear that $b\le c$.
\item

Suppose that the tangent vectors to $a$ and $c$ form a positive basis for $T_{x_{r'}}(S)$. We set $b = a|_{[0,t_1]} \ast (c|_{[s_{r''},s_{r}]})^{-1} \ast a|_{[t_{r'},1]}$. Again, it is clear that $\#(a,b),\#(b,c)<\#(a,c)$. Establishing $a\le b\le c$ involves passing to the universal cover as in the previous sub-case.
\end{enumerate}

\item
Suppose that $x_1=y_1$, and that at $x_1$ the tangent vectors to $a$ and $c$ form a positive basis for $T_{x_1}(S)$. Let $\gamma = c|_{[0,s_1]}\ast (a|_{[0,t_1]})^{-1}$. There are two sub-cases.
\begin{enumerate}

\item
Assume that $\gamma$ is separating in $S\smallsetminus P$. Let $R$ denote the subsurface of $S$ having oriented boundary $\gamma$. Since $a$ and $c$ are assumed to be intersecting efficiently, $R$ can not be a bigon disjoint from $P$. If $R$ contains some point $p'\in P$ it is easy to connect $x_0$ to $p'$ with an arc $b\subset R$ satisfying the desired properties. 

If $R\cap P = \emptyset$, then there is an arc $\eta\subset R$, which starts at $x_0$, runs over a handle in $R$ and ends at $x_1$, and whose interior intersects neither $a$ nor $c$. Set $b = \eta \ast a_{[t_1,1]}$. It is clear that $a\le b\le c$. Although we have not reduced the intersection number, $\#(a,b) = \#(a,c)$, we have reduced to the final sub-case:
\item

Assume that $\gamma$ is not separating. Let $R$ denote the connected component of $S\smallsetminus a\smallsetminus c$ having an oriented boundary component $\gamma$. Since $\gamma$ is not separating, $R$ must have atleast one other boundary component $\delta\ne\gamma$, which consists of a union of sub-arcs along $a$ and $b$ along with possibly $\partial S$.

We claim that for some $i>0$, $a_{[t_i,t_{i+1}]}\subset \delta$. If $\delta$ contains no such sub-arc of $a$ then it must then contain $a|_{[0,t_1]}$ as an oriented sub-arc. It follows that $\delta$ contains $c_{[s_1,s_2]}$. Let $x_i = y_2$, then $\delta$ contains either $(a|_{[t_{i-1},t_{i}]})^{-1}$ or $a|_{[t_i,t_{i+1}]}$, contradicting our assumption.

Now, we may connect $x_0$ to $x_{i+1}$ by an arc $\eta\subset R$ initially to the right of $a$ and left of $c$. Setting $b = \eta\ast a_{[t_{i+1},1]}$ gives the desired arc. 
\end{enumerate}
\end{enumerate}

\end{proof}

We turn to the proof of Theorem \ref{thm:largetwist}. For a pointed monodromy $h$, we let $\widehat{t}(h)$ denote the BRAID invariant of its transverse closure.
\begin{proof}
Let $e_0$ be an embedded arc connecting $\partial S$ to $p\in P$.
By Lemma \ref{lem:boundarytwisting} we have that $c(g)>1 \implies \tau_{\partial S} ^{-1} \circ g$ is right veering; in particular $\tau_{\partial S}^{-1} \circ g$ takes the arc $e_0$ to the right.

By Lemmas \ref{lem:simple} and \ref{lem:sequence} there exists a product, $P$, of right-handed Dehn and half twists taking the arc $e_0$ to $\tau_{\partial S} ^{-1} \circ g(e_0)$. Comultiplication (Theorem \ref{thm:comultiplication}) combined with Proposition \ref{prop:nonzero} implies $\widehat{t}(P)\ne 0$. The composition
\[
P^{-1}\circ \tau_{\partial S} ^{-1} \circ g= \tau_{\partial S} ^{-1} \circ P^{-1} \circ g
\]
fixes the arc $e_0$, and hence a small neighborhood of $\partial S \cup e_0$, which is an annular neighborhood of $\partial S$ containing only the marked point $p$. Lemma \ref{lemma:recognition} then implies that the closure $P^{-1}\circ g$ is transversely isotopic to $B'\cup K'$ for some $K'$ braided about $B'$, the binding of the underlying open book.
By Theorem \ref{thm:binding} we have that $\widehat{t}(P^{-1}\circ g)\ne 0$. A final application of Theorem \ref{thm:comultiplication} coupled with $\widehat{t}(P)\ne 0$ gives that $\widehat{t}(g)\ne 0$.

\end{proof}


\nocite{IKess,IKquasi,VVbind,StipV,LOSS, equiv, Pla, comultcontact, comultgrid, HKM, contactclass, openbooks, legtransverse, torsionob, trivialbraid, holknots}


\thispagestyle{empty}
{\small
\markboth{References}{References}
\bibliographystyle{myalpha}
\bibliography{mybib}{}

\begin{thebibliography}{BVVV13}

\bibitem[Bal08]{comultcontact}
John~A Baldwin.
\newblock {Comultiplicativity of the Ozsv\'{a}th-Szab\'{o} contact invariant}.
\newblock {\em Mathematical Research Letters}, 15(2):273--287, 2008.

\bibitem[Bal10]{comultgrid}
John~A Baldwin.
\newblock {Comultiplication in link Floer homology and transversely non-simple
  links}.
\newblock {\em Algebraic and Geometric Topology}, 10:1417--1436, 2010.

\bibitem[Ben83]{ben}
Daniel Bennequin.
\newblock {Entrelacements et equations de Pfaff}.
\newblock {\em Third Schnepfenried geometry conference}, 1(107):87--161, 1983.

\bibitem[BG15]{trivialbraid}
John Baldwin and J.~Elisenda Grigsby.
\newblock {Categorified invariants and the braid group}.
\newblock {\em Proceedings of the AMS}, 143(7):2801--2814, 2015.

\bibitem[BVVV13]{equiv}
John~A Baldwin, David~Shea Vela-Vick and Vera V\'{e}rtesi.
\newblock {On the equivalence of Legendrian and transverse invariants in knot
  Floer homology}.
\newblock {\em Geometry and Topology}, 17(2):925--974, 2013.

\bibitem[EO08]{monhom}
John Etnyre and Burak Ozbagci.
\newblock {Invariants of contact structures from open books}.
\newblock {\em Transactions of the American Mathematical Society},
  360:3133--3151, 2008.

\bibitem[Etn04]{legtransverse}
John Etnyre.
\newblock {Legendrian and Transversal Knots}.
\newblock {\url{https://arxiv.org/abs/math/0306256}}, 2004.

\bibitem[Etn05]{openbooks}
John Etnyre.
\newblock {Lectures on open book decompositions and contact structures}.
\newblock {\url{https://arxiv.org/abs/math/0409402}}, 2005.

\bibitem[EVV10]{torsionob}
John Etnyre and David~Shea Vela-Vick.
\newblock {Torsion and Open Book Decompositions}.
\newblock {\em International Mathematics Research Notices},
  2010(22):4385--4398, 2010.

\bibitem[FM12]{primer}
Benson Farb and Dan Margalit.
\newblock {\em {A primer on mapping class groups}}, volume~49.
\newblock {Princeton University Press}, 2012.

\bibitem[Gir02]{giroux}
Emmanuel Giroux.
\newblock {Geometrie de contact: de la dimension trois vers les dimensions
  superieures}.
\newblock {\em Proceedings of the International Congress of Mathematicians},
  2:405--414, 2002.

\bibitem[HKM07]{RVD}
Ko~Honda, William~H. Kazez and Gordana Mati\'{c}.
\newblock {Right-veering diffeomorphisms of compact surfaces with boundary}.
\newblock {\em Inventiones mathematicae}, 169:427--449, 2007.

\bibitem[HKM08]{RVD2}
Ko~Honda, William~H. Kazez and Gordana Mati\'{c}.
\newblock {Right-veering diffeomorphisms of compact surfaces with boundary II}.
\newblock {\em Geometry and Topology}, 12:2057--2094, 2008.

\bibitem[HKM09]{HKM}
Ko~Honda, William~H. Kazez and Gordana Mati\'{c}.
\newblock {On the contact class in Heegaard Floer homology}.
\newblock {\em Journal of Differential Geometry}, 83:289--311, 2009.

\bibitem[HP13]{QOB}
Matthew Hedden and Olga Plamenevskaya.
\newblock {Dehn surgery, rational open books and knot Floer homology}.
\newblock {\em Algebraic and Geometric Topology}, 13(3):1815--1856, 2013.

\bibitem[IK17a]{IKess}
Tetsuya Ito and Keiko Kawamuro.
\newblock {Essential open book foliations and fractional Dehn twist
  coefficient}.
\newblock {\em Geometriae Dedicata}, 187(1):17--67, 2017.

\bibitem[IK17b]{IKquasi}
Tetsuya Ito and Keiko Kawamuro.
\newblock {Quasi right-veering braids and non-loose links}.
\newblock {\url{https://arxiv.org/abs/1601.07084}}, 2017.

\bibitem[LOSS09]{LOSS}
Paolo Lisca, Peter Ozsv\'{a}th, Andr\'{a}s~I. Stipsicz and Zolt\'{a}n
  Szab\'{o}.
\newblock {Heegaard Floer invariants of Legendrian knots in contact
  three-manifolds}.
\newblock {\em Journal of the European Mathematical Society}, 11(6):1307--1363,
  2009.

\bibitem[OS03]{orevkov}
S.~Yu. Orekov and V.~V. Shevchishin.
\newblock {Markov Theorem for transversal links}.
\newblock {\em Journal of Knot Theory and Its Ramifications}, 12:905--913,
  2003.

\bibitem[OS04a]{holknots}
Peter Ozsv\'{a}th and Zolt\'{a}n Szab\'{o}.
\newblock {Holomorphic disks and knot invariants}.
\newblock {\em Advances in Mathematics}, 186(1):58--116, 2004.

\bibitem[OS04b]{holdisks3}
Peter Ozsv\'{a}th and Zolt\'{a}n Szab\'{o}.
\newblock {Holomorphic disks and topological invariants for closed
  three-manifolds}.
\newblock {\em Annals of Mathematics}, 2004(159):1027--1158, 2004.

\bibitem[OS05]{contactclass}
Peter Ozsv\'{a}th and Zolt\'{a}n Szab\'{o}.
\newblock {Heegaard Floer homology and contact structures}.
\newblock {\em Duke Mathematical Journal}, 129(1):39--61, 2005.

\bibitem[OST08]{grid}
Peter Ozsv\'{a}th, Zolt\'{a}n Szab\'{o} and Dylan Thurston.
\newblock {Legendrian knots, transverse knots and combinatorial Floer
  homology}.
\newblock {\em Geometry and Topology}, 12:941--980, 2008.

\bibitem[Pav11]{pav}
Elena Pavelescu.
\newblock {Braiding knots in contact 3-manifolds}.
\newblock {\em Pacific Journal of Mathematics}, 253:475--487, 2011.

\bibitem[Pla15]{Pla}
Olga Plamenevskaya.
\newblock {Transverse invariants and right-veering}.
\newblock {\url{https://arxiv.org/abs/1509.01732}}, 2015.

\bibitem[SV09]{StipV}
Andr\'{a}s~I. Stipsicz and Vera V\'{e}rtesi.
\newblock {On invariants for Legendrian knots}.
\newblock {\em Pacific Journal of Mathematics}, 239(1):157--177, 2009.

\bibitem[VV11]{VVbind}
David~Shea Vela-Vick.
\newblock {On the Transverse Invariant for Bindings of Open Books}.
\newblock {\em Journal of Differential Geometry}, 88(3):533--552, 2011.

\bibitem[Wri02]{wrinkle}
Nancy Wrinkle.
\newblock {The Markov theorem for transverse knots}.
\newblock {\em PhD Thesis, Columbia University}, 2002.

\end{thebibliography}
}

\end{document}